\documentclass[11pt,a4paper]{article}

\usepackage{amssymb}
\usepackage[usenames]{color}
\usepackage{theorem}
\usepackage [applemac] {inputenc} 
\usepackage{enumerate}

\newtheorem{theorem}{Theorem}[section]
\newtheorem{proposition}[theorem]{Proposition}
\newtheorem{lemma}[theorem]{Lemma}
\newtheorem{corollary}[theorem]{Corollary}
\newtheorem{definition}[theorem]{Definition}
\newtheorem{remark}[theorem]{Remark}
\newtheorem{problem}[theorem]{Problem}
\newtheorem{exercise}[theorem]{Exercise}

\newenvironment{proof}{\noindent\textbf{Proof}}
                      {$\Box$\vskip\theorempostskipamount}

\begin{document}
 
\title{\textbf{Decision problems for finite and infinite presentations of groups and monoids}}

\author{Carmelo Vaccaro}

\date{}

\maketitle

\begin{abstract} In this survey we show how well known results about the Word Problem for finite group presentations can be generalized to the Word Problem and other decision problems for non-necessarily finite monoid and group presentations. This is done by introducing functions playing the same rôle of the Dehn function for the given decision problem and by finding the Tietze transformations that leave this function invariant. This survey presents some original ideas and points of view.
 \end{abstract}

\smallskip \smallskip

\textit{2000 Mathematics Subject Classification}: Primary 20M05, 20F10; 

Secondary 03D40.

\tableofcontents

  \section*{Introduction}
  
 The Word Problem for finite group presentations is a very well studied problem. The following results are classical: \begin{enumerate} [(a)]
  
  \item \textit{A finite group presentation has a solvable Word Problem if and only if its Dehn function is computable.}
  
  \item \textit{The Dehn functions of two finite presentations for the same group are equivalent.}
  
\end{enumerate}

The goal of this survey is to show how these results can be generalized to the Word, the Conjugacy and the Membership Problem of non-necessarily finite monoid\footnote{this includes also semigroup presentations, see the observation after Definition \ref{sempre}} and group presentations.

The first question to answer is: what are the functions which play for the above decision problems the rôle played by the Dehn function? That is, what is the \textit{decidability function} for any of these problems, i.e., the function whose computability is equivalent to the solvability of the decision problem?

For the Word Problem of finite monoid presentations this function is the \textit{derivational length} (Definition \ref{derlen}), which remarkably was introduced by K. Madlener and F. Otto \cite{MO} long before the Dehn function (introduced by M. Gromov in \cite{Grom}). Part 2 of Proposition \ref{gersho} generalizes (a) to finite monoid presentations.

In a recent paper \cite{GI}, R. I. Grigorchuk and S. V. Ivanov have introduced for decidable (non-necessarily finite) group presentations a function $f_1$ which is computable if and only if the Word Problem is solvable. It has to be noticed that $f_1$ is equivalent to a function (called \textit{work}) introduced earlier by J.-C. Birget \cite{Birg} for monoid and group presentations. In 1 of Proposition \ref{gersho} we prove that the function \textit{work} permits to generalize (a) also to decidable monoid presentations.

It has to be noticed that there are examples of infinite decidable group presentations with unsolvable Word Problem but with a computable Dehn function (see Example 2.4 of \cite{GI}).

It is interesting to consider a group presentation as a monoid presentation and comparing the Dehn function and the derivational length. We have found a quadratic relation between these two functions in the case of finite presentations (2 of Proposition \ref{opti}). In a future paper we will show a linear relation under a mild hypothesis on the presentation, thus proving that in this situation the Dehn function and the derivational length are equivalent.

Let us now talk about the Conjugacy Problem. We have introduced (Definition \ref{gamma}) a function, denoted $\Gamma$, which is computable if and only if the Conjugacy Problem is solvable (on condition that the Word Problem be solvable, see Proposition \ref{CP}). We have treated the Conjugacy Problem in the context of monoid presentations, not only group presentations. For monoids the conjugacy relation is not necessarily symmetric; moreover the solvability of the Conjugacy Problem does not necessarily imply that of the Word Problem. This raises interesting questions about what can be generalized to monoids of the well known results about the conjugacy in groups (see Problems \ref{prob1} and \ref{prob2}).

The decidability function for the Membership Problem is the \textit{distortion function} (Definition \ref{dist}), introduced by Gromov \cite{Grom2}. The generalization of (a) to the Membership Problem (Proposition \ref{MP}) is due to B. Farb \cite{Farb} for group presentations and to Margolis, Meakin and \u{S}uni\'k \cite{MMS} for monoid presentations.

Let us now consider (b). What is special about finite presentations for the same group is that given any two of them, one can be obtained from the other by applications of \textit{elementary} Tietze transformations. Given two non-necessarily finite presentations for the same group this is not the case. To generalize this fact one has to change the point of view and given a finite or infinite presentation $\mathcal{P}$ consider all the presentations obtainable from $\mathcal{P}$ by means of Tietze transformations of a certain kind. A question that naturally arises is: given a decision problem, what are the Tietze transformations that leave its decidability function invariant? To answer this question we have introduced the so-called \textit{bounded Tietze transformations} (see Definition \ref{btt}) and found for any decidability function the kind of bounded Tietze transformations that leave the function invariant up to equivalence (Propositions \ref{dasz} and \ref{new}, Remark \ref{dasz2}).

\section{Words and presentations} \label{intro}

As usual, by a \textit{semigroup} we mean a set equipped with an associative product. A \textit{monoid} is a semigroup possessing an identity element. A \textit{group} is a monoid in which every element has an inverse.

Let $S$ be a semigroup devoid of identity element and let 1 be an element not belonging to $S$. Then we can embed $S$ into the monoid $S\cup \{1\}$ where we set $1s=s1=s$ for every $s \in S\cup \{1\}$. We denote by $S^1$ the monoid $S\cup \{1\}$. If $S$ is a monoid then by $S^1$ we denote $S$ itself.

\begin{definition} \label{} \rm A \textit{congruence} on a semigroup $S$ is an equivalence relation $\sim$ compatible with the product of $S$, that is if $a,x$ and $y$ are elements of $S$ and if $x\sim y$ then $a x\sim a y$ and $x a\sim y a$. In this case the set of equivalence classes of $S$ under $\sim$ is a semigroup, which is a monoid or a group if $S$ is. \end{definition}

A congruence on $S$ is a subset of $S^2$, that is it is the set of pairs $(x,y)$ such that $x\sim y$. The intersection of a family of congruences on $S$ is still a congruence.

\begin{definition} \label{} \rm Let $R\subset S^2$; the congruence \textit{generated by $R$}, denoted $\langle R \rangle$, is the intersection of all the congruences on $S$ containing $R$. We denote by $S/ \langle R \rangle$ the semigroup of the equivalence classes of $\langle R \rangle$. \end{definition}

\begin{remark} \label{raven} \rm If $R$ is \textit{symmetric}, i.e., if $(u,v)\in R$ implies that $(v,u)\in R$, then the congruence generated by $R$ can be described in the following way. Let $u, v \in S$ and let $(a, b)\in R$; we say that $(u, v)$ is a \textit{one step $R$-derivation \big(by means of $(a, b)$\big)} if there exists $x,y \in S^1$ such that $u=xay$ and $v=xby$.

We say that $(u, v)$ is an \textit{$R$-derivation} if $u=v$ or if there exist $a_0=u$, $a_1$, $\cdots$, $a_k=v \in S$ such that $(a_{i-1}, a_i)$ is a one step $R$-derivation for every $i=2, \cdots, k$. In this case we say that $(u, a_1, \cdots, a_{k-1}, v)$ is a \textit{$k$ steps $R$-derivation}.

 Now we prove that $u$ is congruent to $v$ in the congruence generated by $R$ if and only if $(u, v)$ is an $R$-derivation. Indeed let $\sim$ be the relation defined by: $u\sim v$ if $(u, v)$ is an $R$-derivation. Then $\sim$ is a congruence containing $R$, that is $\sim$ contains $\langle R \rangle$. On the other side if a congruence contains $R$ then it contains any $R$-derivation, that is it contains $\sim$; therefore $\langle R \rangle$ contains $\sim$. \end{remark}

\begin{definition} \label{} \rm Given $R\subset S^2$, we define the \textit{symmetrized of $R$} as the set $R':=R\cup \{(y,x) : (x,y)\in R \}$. \end{definition} 

Obviously $R'$ is symmetric and the congruence generated by $R$ coincides with that generated by $R'$. Thus it is not restrictive to consider only congruences generated by symmetric sets.

\begin{remark} \label{only} \rm Let $S$ be a semigroup (with or without identity), let 1 be an element not belonging to $S$, set $M:=S\cup \{1\}$ and extend to $M$ the product of $S$ by setting $m1=1m=m$ for $m\in M$. We have that $M$ is a monoid whose identity is $1$ (if $S$ is a monoid with identity $e$ then $e$ is not the identity of $M$ since we have set $e1=e\neq 1$).

Let $R\subset S^2$ and denote by $\langle R \rangle_S$ and $\langle R \rangle_M$ the congruences of $S$ and $M$ respectively generated by $R$. We show that $S/\langle R \rangle_S$ is isomorphic to $(M/\langle R \rangle_M) \setminus{\{1\}}$. 

It is obvious that if $u$ and $v$ are elements of $M$ such that $uv=1$ then necessarily $u=v=1$. Since for every $(a, b) \in R$ we have that $a, b\neq 1$, then if $(xay, xby)$ is a one step $R$-derivation of $M$ then $xay, xby \neq 1$. This implies that the congruence class of 1 contains only 1. This implies also that the $R$-derivations of $S$ and $M$ coincide and that if $u \in S$ then the congruence classes of $u$ in $S$ and in $M$ coincide. Thus the application from $S/\langle R \rangle_S$ to $(M/\langle R \rangle_M) \setminus{\{1\}}$ sending the congruence class of $u$ in $S$ to that in $M$ is well defined and is an isomorphism of semigroups. \end{remark}

\begin{remark} \label{} \rm Let $G$ be a group and let $R\subset G^2$. Set
       $$R_0:=\{uv^{-1} : (u,v)\in R\};$$
   then $\langle R \rangle$ coincides with the congruence generated by the set $\{(w,1) : w \in R_0\}$. Moreover if $R$ is symmetric then $R_0$ contains the inverse of any of its elements. 
   
   Let $\mathcal{N}$ be the normal subgroup of $G$ normally generated by $R_0$, that is the intersection of all the normal subgroups of $G$ containing $R_0$. Then $u$ is congruent to $v$ if and only if $u v^{-1} \in \mathcal{N}$. Moreover\footnote{we observe that $\langle R \rangle$ is a subset of $G^2$ while $\mathcal{N}$ is  a subset of $G$} $G/\langle R \rangle=G/\mathcal{N}$ and $\mathcal{N}$ coincides with the equivalence class of 1 in the congruence generated by $R$. 
   
   Let $R$ be symmetric; then there is an $R$-derivation from $u$ to $v$ if and only if there exist $a_1, \cdots, a_k \in G$ and $r_1, \cdots, r_k \in R$ such that $u v^{-1}=a_1 r_1 a_1^{-1} \cdots a_k r_k a_k^{-1}$.  \end{remark}

\begin{definition} \label{} \rm Let $X$ be a set (finite or infinite); the \textit{free monoid} on $X$, denoted $\mathcal{M}(X)$, is the set of words on $X$ equipped with the usual operation of concatenation of words. The identity element of $\mathcal{M}(X)$ is the word with zero letters, denoted 1. If $w=x_1 \cdots x_m$ is an element of $\mathcal{M}(X)$ with the $x_i \in X$ then the \textit{length} of $w$, denoted $|w|$, is the natural number $m$. The length of 1 is zero. The elements of $X$ are called \textit{letters}. Given $(u,v)\in \mathcal{M}(X)^2$ we define the \textit{length of $(u,v)$} as $|u|+|v|$. \end{definition}

$\mathcal{M}(X)$ is \textit{free on the set $X$ in the category of monoids} (see Def. I.7.7 of \cite{Hunger}), that is for every monoid $M$ and for every function $f$ from $X$ to $M$ there is one and only one homomorphism from $\mathcal{M}(X)$ to $M$ which extends $f$.

\begin{definition} \label{} \rm The \textit{free semigroup} on $X$, denoted $\mathcal{M}(X)_+$, is equal to $\mathcal{M}(X)$ minus the element 1, that is it is the set of words on $X$ of positive length. It is free on the set $X$ in the category of semigroups. \end{definition}

We recall that $\langle R \rangle$ denotes the congruence generated by $R$.

\begin{definition} \label{semo} \rm Let $R$ be a subset of $\mathcal{M}(X)^2$; we say that $<  X \, | \, R  >$ is a \textit{monoid presentation} for $\mathcal{M}(X)/\langle R \rangle$. A monoid presentation is also called a \textit{string rewriting system} or a \textit{Thue system}. The elements of $X$ are called \textit{generators} and those of $R$ are called \textit{defining relations}. The elements of $\langle R \rangle$ are called \textit{relations}.  \end{definition}

\begin{definition} \label{sempre} \rm Let $R$ be a subset of $\mathcal{M}(X)_+^2$; we say that $<  X \, | \, R  >$ is a \textit{semigroup presentation} for $\mathcal{M}(X)_+/\langle R \rangle$.    \end{definition}

By virtue of Remark \ref{only}, if $\mathcal{P}=<  X \, | \, R  >$ is a semigroup presentation for the semigroup $S$ then $\mathcal{P}$ considered as a monoid presentation presents\footnote{we observe that if $S$ is a monoid and if $e$ is the identity of $S$ then as seen in Remark \ref{only}, $e$ is not the identity for $S\cup\{1\}$} $S\cup\{1\}$ where 1 is an element not belonging to $S$ and is the identity element for $S\cup\{1\}$. 

This means that it is not restrictive to consider only monoid presentations; indeed, if $\mathcal{P}$ is a semigroup presentation for the semigroup $S$, then we can consider $\mathcal{P}$ as a monoid presentation for a monoid $M$ and retrieve $S$ by eliminating the element 1 from $M$.

\begin{definition} \label{stop} \rm Let $X^{-1}$ be a set disjoint from $X$ such that $|X| = |X^{-1}|$ and suppose given a bijection $X\rightarrow X^{-1}$. We denote by $x^{-1}$ the image by this bijection of an element $x \in X$ and we call it \textit{the inverse of x}. If $y \in X^{-1}$ we denote by $y^{-1}$ the unique element of $X$ such that $(y^{-1})^{-1}=y$. The \textit{free group} on $X$, denoted $\mathcal{F}(X)$, is the quotient of the free monoid $\mathcal{M}(X \cup X^{-1})$ by the congruence generated by 
                    $$\texttt{F}:=\big\{(xx^{-1}, 1) : x \in X \cup X^{-1}\big\},$$
      that is $< X \cup X^{-1} \, | \, \texttt{F} >$ is a monoid presentation for $\mathcal{F}(X)$.  \end{definition}

$\mathcal{F}(X)$ is free on the set $X$ in the category of groups.

Any equivalence class of the congruence generated by $\texttt{F}$ contains one and only one \textit{reduced word}, i.e., a word of the form $x_1 \cdots x_m$ such that $x_i^{-1}\neq x_{i+1}$ for every $i=1, \cdots, m-1$ (see The. 1.2 of \cite{MKS}). If $u\in \mathcal{M}(X \cup X^{-1})$ and if $u'$ is the unique reduced word congruent to $u$ then we say that $u'$ is the \textit{reduced form} of $u$.

\begin{definition} \label{rho} \rm We let $\rho: \mathcal{M}(X \cup X^{-1}) \rightarrow \mathcal{F}(X)$ be the function sending a word to its unique reduced form.  \end{definition}

We can consider $\mathcal{F}(X)$ as a subset of $\mathcal{M}(X \cup X^{-1})$ (but not as a subgroup). Given two words $u$ and $v$ of $\mathcal{F}(X)$, we denote by $u v$ their product in $\mathcal{M}(X \cup X^{-1})$ and by $\rho(u v)$ that in $\mathcal{F}(X)$. If $u v$ is reduced then the products of $u$ by $v$ in $\mathcal{M}(X \cup X^{-1})$ and in $\mathcal{F}(X)$ are equal.

\begin{definition} \label{segr} \rm  Let $R$ be a subset of $\mathcal{F}(X)$; we say that $<  X \, | \, R  >$ is a \textit{group presentation} for $\mathcal{F}(X)/\mathcal{N}$, where $\mathcal{N}$ is the normal subgroup of $\mathcal{F}(X)$ normally generated by $R$. The elements of $R$ and those of $\mathcal{N}$ are called respectively \textit{defining relators} and \textit{relators}. 

We say that $<  X \, | \, R  >$ \textit{corresponds} to the monoid presentation $< X \cup X^{-1} \, | \, R_1 >$ where $R_1=\{(u,v^{-1}) \in \mathcal{F}(X)^2: uv \in R\} \cup \texttt{F}$. 
     \end{definition}

For every $w \in R$ there are exactly $|w|+1$ pairs of words $(x,y)$ such that $xy=w$; thus if $R$ is finite then $|R_1|=|R|+\sum_{w\in R}|w|$ .

\begin{definition} \label{comm} \rm  Let $X$ be a set; the \textit{free commutative monoid on $X$} \big(denoted $\mathcal{CM}(X)$\big) is the monoid presented by the monoid presentation
         $$<  X \, | \, (xy, yx) \, \,\mathrm{for \, \, every \, \,} x, y \in X >.$$
     The \textit{free abelian group on $X$} \big(denoted $\mathcal{FA}(X)$\big) is the group presented by the group presentation
     $$<  X \, | \, xyx^{-1}y^{-1} \, \,\mathrm{for \, \, every \, \,} x, y \in X >.$$
      \end{definition}

 \section{Computable functions and decidable sets} \label{}

For a simple but rigorous introduction to \textit{computable functions} and \textit{decidable sets} a good reference is the book of Shen-Vereshchagin \cite{Shen}.

\begin{definition} \rm \label{} Let $E$ be a subset of $\mathbb{N}$ and let $f$ be a function from $E$ to $\mathbb{N}$. The function $f$ is said \textit{computable} if there exists an algorithm which taken as input an $n\in \mathbb{N}$ then \begin{itemize}

  \item it halts and gives $f(n)$ as output if $n\in E$;

  \item it does not halt if $n\notin E$.
 
\end{itemize}
     \end{definition}

A computable function is also called \textit{recursive}.

\begin{definition} \label{} \rm The set $E\subset \mathbb{N}$ is said \textit{decidable} if its characteristic function is decidable, i.e., if there exists an algorithm which determines whether an arbitrary $n\in \mathbb{N}$ belongs to $E$.
     \end{definition}

 \begin{definition} \label{enume} \rm  The set $E$ is said \textit{enumerable} if there exists an algorithm which enumerates the elements of $E$. This means that there is an algorithm which for every $n$ outputs a certain $e_n$ and that $E=\{e_n : n \in \mathbb{N}\}$. We assume that for some $n$ the output of this algorithm can be empty, that is to be formal we must write $E=\{e_n : n \in \mathbb{N} \,\, \textrm{and} \, \, e_n \, \, \textrm{is not empty}\}$. We also assume that $e_m$ can be equal to $e_n$ for $m\neq n$ and that we are able to determine whether or not $e_m=e_n$.
     \end{definition}

A decidable set is enumerable. It is obvious that $E$ is decidable if and only if $E$ and $\mathbb{N} \setminus{E}$ are enumerable. Decidable and enumerable sets are also called \textit{recursive} and \textit{recursively enumerable} respectively in the literature.

 Any finite subset of $\mathbb{N}$ is decidable (see 1.2 of \cite{Shen}). Anyway when talking about a finite set $E$ we do not only assume that its elements are finitely many but we require also the following condition: \textit{there exists an algorithm enumerating $E$ which after a finite time it gives only empty outputs\footnote{or equivalently we can say that after a finite time the algorithm stops and does not output anything more}}. This condition is necessary to avoid paradoxes (see Appendix \ref{app}). 
 
 We call a set verifying the latter condition a \textit{finite and effectively computable set}. It turns out that almost every finite set encountered in mathematics is also effectively computable, so we will mostly use the the terms \textit{finite set} and \textit{finite and effectively computable set} as synonymous.

 \begin{definition} \rm \label{}  Let $U$ and $V$ be enumerable sets, that is $U=\{u_n : n \in \mathbb{N}\}$ and $V=\{v_n : n \in \mathbb{N}\}$, where there are two algorithms whose $n$-th outputs are $u_n$ and $v_n$ respectively. Let $V'$ be a subset of $V$ and let $f$ be a function from $V'$ to $U$. Let $\overline{f}$ be the following function: let $n$ be a natural such that $v_n \in V'$ and let $f(v_n)=u_m$; then set $\overline{f}(n)=m$. Let $E=\{n \in \mathbb{N} : v_n \in V'\}$; then $\overline{f}$ is a function from $E$ to $\mathbb{N}$. We say that $f$ is \textit{computable} if $\overline{f}$ is computable, that is if there exists an algorithm which taken $v_n$ as input, it halts and gives $f(v_n)$ as output if $v_n \in V'$, otherwise it does not halt. 
 
 In an analogous way one can define \textit{decidable} and \textit{enumerable} subsets of an enumerable set.
      \end{definition}

   \begin{remark} \label{yroth}  \rm If $E$ and $F$ are enumerable sets, then also $E\times F$ is. More generally if $\{E_n\}_{n\in \mathbb{N}}$ is an enumerable family of enumerable sets, then also $\bigcup_{n\in \mathbb{N}} E_n$ is enumerable. This proves that if $X$ is enumerable then the free monoid $\mathcal{M}(X)$ and the free group $\mathcal{F}(X)$ are enumerable. Thus a monoid or group with an enumerable set of generators is enumerable. 
   
   The same is true also for countable sets, that is a monoid or group generated by a countable set of elements is countable. We recall that by admitting the \textit{Church-Turing thesis} (see 3.3 of \cite{Sipser}), there are countable sets which are not enumerable since there are countably many Turing machines and uncountably many subsets of $\mathbb{N}$.
      \end{remark}

 \begin{definition} \rm \label{}  A presentation $\mathcal{P}=<  X \, | \, R  >$ is said \textit{finitely} or \textit{decidably} or \textit{enumerably generated} if $X$ is respectively finite, decidable or enumerable. In the same way $\mathcal{P}$ is said \textit{finitely} or \textit{decidably} or \textit{enumerably related} if $R$ is finite, decidable or enumerable. A presentation which is finitely generated and related is said \textit{finite}. A presentation which is finitely generated and decidably related is said \textit{finitely generated decidable}.
      \end{definition}

If a group presentation is finitely, decidably or enumerably generated or related, the same is true for the monoid presentation corresponding to it (Definition \ref{segr}).

   \begin{proposition} \label{dger}  If a monoid admits a decidably generated and enumerably related presentation then it admits also a presentation which is decidably generated and related. \end{proposition}

   \begin{proof}  Indeed let $\mathcal{P}=<  X \, | \, R  >$ be a presentation with $X$ decidable and $R$ enumerable and let $M$ be the monoid presented by $\mathcal{P}$. Thus $R=\{(a_n, b_n) : n \in \mathbb{N}\}$ where there is an algorithm whose $n$-th output is $(a_n, b_n)$. Let $e$ be an element not belonging to $\mathcal{M}(X)$, set
    $$R_1:=\big\{(ex, x) : x\in X\cup\{e\}\big\} \cup \{(xe, x) : x\in X\}$$
and let
      $$R'=\{(e^n a_n, b_n) : n \in \mathbb{N}\} \cup R_1.$$
    Obviously $< X\cup \{e\} \, | \, R' >$ is a decidably generated presentation of $M$, which is finitely generated if $\mathcal{P}$ is finitely generated. Let us prove that $R'$ is decidable. Let $(u, v)$ be an element of $\mathcal{M}(X)^2$. Since $X$ is decidable, the same is true for the set $R_1$, thus one can decide whether or not $(u, v)$ belongs to $R_1$. Let $(u, v)$ do not belong to $R_1$. If $u$ is not of the form $e^m w$ where $w$ is a word not containing $e$ then $(u, v)$ does not belong to $R'$; the same if $v$ contains the letter $e$. Otherwise let us compute $(a_m, b_m)$; if $w=a_m$ and $v=b_m$ then $(u, v)$ belongs to $R'$ otherwise it does not. \end{proof}

By the proof of Proposition \ref{dger} we have the following

     \begin{corollary}  \label{dger2}  If a monoid admits a finitely generated and enumerably related presentation then it admits also a presentation which is finitely generated and decidably related. \end{corollary}

Proposition \ref{dger} and Corollary \ref{dger2} are valid also for group presentations.

 \begin{proposition} \label{} Let $\mathcal{P}=<  X \, | \, R  >$ be a monoid presentation and let $\mathcal{M}(X)$ and $R$ be enumerable. Then the set of derivations for $\mathcal{P}$ is enumerable.   \end{proposition}

 \begin{proof} First we prove that for every $k$, the set of $k$-steps derivations is enumerable. Let $k=1$; the set of 1-step derivations is enumerable since it is the set of pairs $(xay, xby)$ where $x, y \in \mathcal{M}(X)$ and $(a,b) \in R$, thus it can be represented as the Cartesian product $\mathcal{M}(X)^2 \times R$, which is enumerable since $\mathcal{M}(X)$ and $R$ are enumerable. 
 
 By induction hypothesis the sets of 1-step derivations and of $(k-1)$-steps derivations are  enumerable, thus their Cartesian product is enumerable and there is an algorithm $\texttt{A}$ enumerating it. Then we define an algorithm $\texttt{A}_k$ in the following way: let $n$ be a natural number and let $(a_0, a_1, \cdots, a_k, b, c)$ be the $n$-th output of $\texttt{A}$, where $(a_0, a_1, \cdots, a_{k-1})$ is a $(k-1)$-steps derivations and $(b, c)$ a 1-step derivation. Then if $a_{k-1}=b$, the $n$-th output of $\texttt{A}_k$ is $(a_0, a_1, \cdots, a_k, c)$, otherwise it is an empty output. The algorithm $\texttt{A}_k$ enumerates the set of $k$ steps derivations.

 The set of derivations of $\mathcal{P}$ is the union for all the naturals $k$ of the sets of $k$ steps derivations and by Remark \ref{yroth} is enumerable since each of these is enumerable.  \end{proof}

 We have seen in Remark \ref{yroth} that if $X$ is enumerable then $\mathcal{M}(X)$ too is enumerable. If $X$ and $R$ are enumerable then we can suppose that the algorithms enumerating $X$ and $R$ operate in parallel and that when an element $r=x_1 \cdots x_m\in R$ is outputted with the $x_i\in X$, then the elements $x_1, \cdots, x_m$ have already been outputted.
 
  The set of relations of $\mathcal{P}$ is the set of pairs $(a, b)$ such that there is a derivation from $a$ to $b$. Therefore if the set of derivations is enumerable also that of relations is enumerable. Thus we have the following

 \begin{corollary} \label{avig} Let $\mathcal{P}=<  X \, | \, R  >$ be a presentation and let $X$ and $R$ be enumerable; then $\mathcal{M}(X)$, the set of derivations and the set of relations of $\mathcal{P}$ are enumerable.   \end{corollary}

 \section{Van Kampen diagrams for monoid presentations}  \label{VKdfsp}
 
 Van Kampen diagrams for group presentation are well studied in the literature (see for instance Sec. V.1 of \cite{LS}, \S 11 of \cite{Olsh} or Sec. 4 of \cite{Brids}). In \cite{Rem1}, J. H. Remmers defined diagrams for semigroup presentations and generalized to semigroups the well known results about small cancellations presentations (see also \cite{Rem2} and Sec. 5 of \cite{Higg}).

For full details about the notions treated in this section we refer the reader to 1.7 of \cite{Higg}.

\smallskip
 
 A \textit{source} in a directed graph is a vertex with no \textit{out-edges}, a \textit{sink} is one with no \textit{in-edges}. If $X$ is a set then an \textit{$X$-labeled graph} is a graph with an application from its edges to $X$. The label of a path is the word on $X$ obtained by the concatenation of the labels of its edges from the first to the last.

  In a planar 2-cell complex every edge can belong to at most two faces (we do not consider the unbounded exterior region as a face). The \textit{boundary} of a planar 2-cell complex is the set of edges which do not belong to two faces, that is they belong to only one face or to no one. These edges are said \textit{exterior} and a path containing only exterior edges is said an \textit{exterior path}.

Let $\mathcal{P}=< X \, | \, R >$ be a monoid presentation. A \textit{van Kampen diagram} for $\mathcal{P}$ is a planar and simply connected $X$-labeled 2-cell complex $\mathcal{C}$ such that \begin{enumerate}

  \item $\mathcal{C}$ has exactly one source and one sink;
  
  \item either there is a simple cycle going from the source to the sink of $\mathcal{C}$; or there are two exterior paths going from the source to the sink of $\mathcal{C}$ whose intersection does not contain any edge belonging to a face;
  
  \item the subgraph constituted by a face of $\mathcal{C}$: either is a simple cycle and if $u$ is its label then $(u, 1)$ is a defining relation; or has exactly one source and one sink and has exactly two simple paths going from the source to the sink and if $u$ and $v$ are the labels of these paths then $(u, v)$ is a defining relation. 
    \end{enumerate}

We observe that there can be more than one pair of exterior paths in $\mathcal{C}$ verifying 2, but we assume that we have fixed one of these pairs. If $u$ and $v$ are the labels of these exterior paths then we say that $\mathcal{C}$ is a \textit{van Kampen diagram for $(u,v)$}.

    We have the following result (see The. 3.2 of \cite{Rem2} or The. 1.7.2 of \cite{Higg}):
    
\begin{theorem} \label{remhig} Let $\mathcal{P}=< X \, | \, R >$ be a monoid presentation, let $u$ and $v$ be words on $X$ and let $k$ be a natural number. Then there is a $k$ steps derivation from $u$ to $v$ if and only if there is a van Kampen diagram for $(u,v)$ with $k$ faces. \end{theorem}

We will prove one direction of the theorem. Let $(u, u_1, \cdots, u_k)$ be a derivation. We show how to associate with it in a standard way a van Kampen diagram in such a way that if $k\geqslant 1$ then there are two exterior simple paths going from the source to the sink of $\mathcal{C}$ and they are labeled by $u$ and $u_k$.

 Let $k=0$, that is we have the trivial relation $(u,u)$. In this case if $u=x_1 \cdots x_m$, we consider the graph which is a simple path with $m$ edges each labeled by $x_i$ for $i=1, \cdots, m$,
 
 \begin{picture}(180,55)(6,14)


\put(135,40){\line(1,0){88}} 


\put(138,37){\vector(1,0){9}}  

\put(135,40){\circle{4}}   

\put(135,40){\circle*{2}}  

\put(158,40){\circle*{2}}   

\put(223,40){\circle*{2}}  

\put(223,40){\circle{4}}   

\put(183,40){\circle*{2}}   



\put(197,40){\circle*{2}}  

\put(143,45){$x_1$}


\put(165,45){$x_2$}


\put(202,45){$x_m$}


%

\end{picture}

 with the arrow indicating the direction of the path. Suppose to have associated a 2-cell complex $\mathcal{C}$ with the derivation $(u, u_1, \cdots, u_{k-1})$. Let $u_{k-1}=v a w$ and $u_k=v b w$ with $(a, b)$ a defining relation. By induction hypothesis we can assume that we have fixed an exterior path of $\mathcal{C}$ labeled $u_{k-1}$ going from the source to the sink of $\mathcal{C}$. 
 
 Let $a=x_1 \cdots x_m$, $b=y_1 \cdots y_n$ and let $F$ be the following 2-cell complex:

 
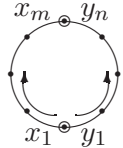
\begin{picture}(180,100)(0,-7)

\put(180,40){\circle{40}}

\put(180,20){\circle{4}} 

\put(180,60){\circle{4}} 

\put(165,15){$x_1$}

\put(186,15){$y_1$}

\put(161,62){$x_m$}

\put(186,62){$y_n$}

\put(177,37){\oval(25,25)[bl]}  

\put(165,37){\vector(0,1){4}} 

\put(184,37){\oval(25,25)[br]}  

\put(196,37){\vector(0,1){4}} 

\put(166,26){\circle*{2}}

\put(160,40){\circle*{2}}

\put(200,40){\circle*{2}}

\put(180,60){\circle*{2}}

\put(180,20){\circle*{2}}

\put(194,54){\circle*{2}}

\put(194,26){\circle*{2}}

\put(166,54){\circle*{2}}



\end{picture}

 Then we associate with the derivation $(u, \cdots, u_k)$ the complex obtained by \textit{adjoining} (see \cite{adj}) $F$ to $\mathcal{C}$ \textit{along} the bijection between the exterior paths of $F$ and $\mathcal{C}$ labeled by $a$, that is by “gluing” $F$ and $\mathcal{C}$ along these paths.

 If $a$ or $b$ are equal to 1, then the source and the sink of $F$ do coincide and the boundary of $F$ is a simple cycle. In particular if $a=1$ then $F$ is glued to $\mathcal{C}$ along a single vertex; if $b=1$ then the two vertices of $\mathcal{C}$ which are the initial and the final vertex of the exterior path labeled by $a$ are made to coincide and a new face is created. 
 
 \smallskip

Let $\mathcal{P}=< X \, | \, R >$ be a group presentation and let $\mathcal{P}'$ be the monoid presentation corresponding to $\mathcal{P}$. Then the van Kampen diagrams for $\mathcal{P}$ are as those associated with $\mathcal{P}'$ but with the following modifications: \begin{itemize}

  \item we consider only diagrams associated with relations of the form $(1,v)$, that is we consider only derivations whose initial word is $u=1$;
  
  \item there is no face for a defining relation $(xx^{-1},1)$, that is in this case the two edges labeled by $x$ and $x^{-1}$ are made to coincide and thus identified;
  
  \item if a defining relation $(xx^{-1},1)$ is applied and the subpaths labeled by $x$ and $x^{-1}$ to be identified are cycles then the portions of the graph comprised in these cycles are eliminated in order to keep the graph planar.

\end{itemize}

For further details about van Kampen diagrams for group presentations see the references cited at the beginning of the section.

\section{The Word Problem}
   
This section deals with the word problem for monoid presentations. Deeper results, but valid only for group presentations, can be found in \cite{VaccAlgComb}.

 \begin{definition} \rm \label{}   Let $\mathcal{P}=<  X \, | \, R  >$ be a presentation for the monoid $M$. We say that $\mathcal{P}$ \textit{has a solvable Word Problem} if given $u, v\in \mathcal{M}(X)$ there exists an algorithm which decides whether the elements of $M$ represented by $u$ and $v$ are equal. This is equivalent to say that the set of relations of $\mathcal{P}$ is decidable. 
 
 We say that a presentation \textit{has a solvable Word Search Problem} if it has a solvable Word Problem and for any relation $(u, v)$ there exists an algorithm which finds a derivation from $u$ to $v$.  
 
 The same definition holds for group presentations with $\mathcal{M}(X)$ replaced by $\mathcal{F}(X)$. \end{definition}   
   
We observe that if $\mathcal{P}$ is a group presentation then $\mathcal{P}$ has a solvable Word Problem if and only if the same is true for the monoid presentation corresponding to $\mathcal{P}$ as in Definition \ref{segr}. Thus all the results which we will prove in the case of monoid presentations hold also for group presentations with the obvious modifications.

Moreover if $\mathcal{P}$ is a group presentation for the group $G$ then we have the equality $u=v$ in $G$ if and only if $uv^{-1}=1$ in $G$, that is the Word Problem is solvable for a group presentation if and only if the set of relators of $\mathcal{P}$ is decidable.

We recall that $\langle R \rangle$ denotes the congruence generated by $R$.

  \begin{proposition} \label{} Let $\mathcal{P}=<  X \, | \, R  >$ be a monoid presentation and let $\mathcal{M}(X)$ be enumerable (in particular let $X$ be enumerable). Then the Word Problem for $\mathcal{P}$ is solvable if and only if the monoid $\mathcal{M}(X)/\langle R \rangle$ is enumerable.  \end{proposition}

\begin{proof} We have that $\mathcal{M}(X)=\{w_n : n \in \mathbb{N}\}$, where there is an algorithm whose $n$-th output is $w_n$. Consider the algorithm whose $n$-th output is the congruence class of $w_n$; by Definition \ref{enume} this algorithm enumerates $\mathcal{M}(X)/\langle R \rangle$ if and only if we can determine whether two elements of $\mathcal{M}(X)/\langle R \rangle$ are equal or not, that is if and only if the Word Problem for $\mathcal{P}$ is solvable.  \end{proof}

 From now on we assume that the set of defining relators of any presentation is symmetric. By Remark \ref{raven} we have that if $\mathcal{P}=<  X \, | \, R  >$ is a presentation for the monoid $M$ and if $u, v \in \mathcal{M}(X)$, then $u=v$ in $M$ if and only if there is an $R$-derivation from $u$ to $v$.

 \begin{remark} \label{spesp}  \rm  Let $\mathcal{P}=< X\, | \, R >$ be a presentation such that the set of $R$-derivations is enumerable (in particular let $X$ and $R$ be enumerable, see Corollary \ref{avig}); then the solvability of the Word Problem is equivalent to that of the Word Search Problem. Indeed let $(u, v)$ be a relation; then the algorithm enumerating the $R$-derivations will give after a finite time an output which is a derivation from $u$ to $v$.

Anyway the complexity of the Word Search Problem can be much greater than that of the simple Word Problem. Madlener and Otto showed in (\cite{MO}, Cor. 6.14) group presentations with Word Problem solvable in polynomial time and with Word Search Problem\footnote{Madlener and Otto call \textit{pseudo-natural algorithm} an algorithm for solving the Word Search Problem} arbitrarily hard. Namely for every $m\geqslant 3$ they constructed a group $G(m)$ with these properties: the complexity of the Word Problem for $G(m)$ is at most polynomial, that of the Word Search Problem is bounded above by a function in the \textit{Grzegorczyk class} $E_m$ (see Ch. 12 of \cite{EC}) but by no function in the Grzegorczyk class $E_{m-1}$. \end{remark}

We end this section with a definition of \textit{equivalence} of functions often used in Combinatorial Group Theory.

\begin{definition} \label{equiv} \rm Let $p$ be a non-zero natural number and let $f, g: \mathbb{R}^p_+ \rightarrow \mathbb{R}^p_+$ be two non-decreasing\footnote{a function $f: \mathbb{R}^p_+ \rightarrow \mathbb{R}^p_+$ is non-decreasing if $f(x_1, \cdots, x_p)\leqslant f(y_1, \cdots, y_p)$ when $x_i \leqslant y_i$ for every $i=1, \cdots, p$} functions. We write $f \preceq g$ if there exists a positive constant $\alpha$ such that 
          $$f(x_1, \cdots, x_p)\leqslant \alpha g(\alpha x_1, \cdots, \alpha x_p) + \alpha (x_1+ \cdots+ x_p)$$
   for every non-zero natural numbers $x_1, \cdots, x_p$. We say that $f$ and $g$ are \textit{equivalent} if $f \preceq g$ and $g \preceq f$ and in this case we write $f \simeq g$.
 
Let $f, g: \mathbb{R}^p_+ \rightarrow \mathbb{R}^p_+$ be two functions. We write $f \preceq_{\texttt{s}} g$ if there exists a positive constant $\alpha$ such that 
     $$f(x_1, \cdots, x_p)\leqslant \alpha g(x_1, \cdots, x_p)$$
    for every non-zero natural numbers $x_1, \cdots, x_p$. 
   We say that $f$ and $g$ are \textit{strongly equivalent} if $f \preceq_{\texttt{s}}g$ and $g \preceq_{\texttt{s}} f$ and in this case we write $f \simeq_{\texttt{s}} g$. \end{definition}

If $f \leqslant g$ then $f \preceq_{\texttt{s}}g$; if $f \preceq_{\texttt{s}}g$ then $f \preceq g$. If $f$ is a function from $\mathbb{N}^p_+$ to $\mathbb{R}^p_+$ then we consider $f$ as defined on $\mathbb{R}^p_+$ by assigning the value $f(n_1, \cdots, n_p)$ to every $(x_1, \cdots, x_p)$ such that $x_i \in ]n_i, n_i+1[$ for $i=1, \cdots, p$.

 \section{Derivational length and derivational work}

  \begin{definition} \label{derlen}  \rm Let $\mathcal{P}=< X\, | \, R >$ be a monoid presentation and let $(u,v)$ be a relation of $\mathcal{P}$. The \textit{derivational length of $(u,v)$} is defined as
       $$\textsf{dl}(u,v)=\textsf{min}\{k\in \mathbb{N} : \mathrm{there \, \,  is \,  \,  a} \, \, k \, \, \mathrm{steps} \, \, R\mathrm{-derivation \, \, from} \, \, u \, \, \mathrm{to} \, \, v\}.$$   
      Let $\mathcal{P}$ be finitely generated and let $n_1$ and $n_2$ be natural numbers; the \textit{derivational length of $\mathcal{P}$ at $(n_1, n_2)$} is
   $$\textsf{DL}(n_1, n_2):=\textsf{max}\{\textsf{dl}(u_1,u_2): (u_1, u_2)   \, \, \mathrm{is \, \,  a \, \, relation},  \,|u_1|\leqslant n_1, \,  |u_2|\leqslant n_2\}.$$  \end{definition}

 Since we always assume that $R$ is symmetric then $\textsf{dl}(v,u)=\textsf{dl}(u,v)$. Moreover we have that $\textsf{dl}(u, u)=0$.
By Theorem \ref{remhig} we have that $\textsf{dl}(u,v)$ is the minimal number of faces of a van Kampen diagram for $(u,v)$.

 The derivational length was introduced by K. Madlener and F. Otto in (\cite{MO}, Sec. 3) with the name of \textit{derivational complexity}.

   \begin{remark} \label{} \rm Obviously $\textsf{DL}(n_2, n_1)=\textsf{DL}(n_1, n_2)$. If $n'_1\leqslant n_1$ and $n'_2\leqslant n_2$ then $\textsf{DL}(n'_1, n'_2) \leqslant \textsf{DL}(n_1, n_2)$.

If $\mathcal{P}$ is the monoid presentation obtained from a group presentation then 
      $$\textsf{dl}(uv^{-1},1) \leqslant \textsf{dl}(u,v) + |v|$$
and thus 
     $$\textsf{DL}(m+ n, 0) \leqslant \textsf{DL}(m, n) +\textsf{min}(m,n).$$
 Indeed if $(u, u_1, \cdots, u_{k-1}, v)$ is a $k$ steps derivation then also
            $$(uv^{-1}, u_1v^{-1}, \cdots, u_{k-1}v^{-1}, vv^{-1})$$
  is a $k$ steps derivation and obviously there is a derivation of length $|v|$ from $vv^{-1}$ to 1.
    \end{remark}

       \begin{definition} \label{omega} \rm Let $\mathcal{P}:=< X\, | \, R >$ be a monoid presentation and let $(u,v)$ be a defining relation. The \textit{(derivational) work of $(u,v)$} is
                $$\textsf{work}(u,v)=|u|+|v|.$$          
Let $(a, b)$ be a 1-step derivation by means of the defining relation $(u,v)$, that is there exist words $a_1, a_2, b_1$ and $b_2$ such that $a=a_1 u a_2$ and $b=b_1 v b_2$. The \textit{work of $(a,b)$} is defined as the work of $(u,v)$. The \textit{(derivational) work of} a $k$ steps derivation $(u_0, u_1, \cdots, u_k)$ is defined as 
                         $$\sum_{i=1}^k \textsf{work}(u_{i-1}, u_i).$$
Let $(u_1, u_2)$ be a relation. The \textit{(derivational) work of $(u_1, u_2)$} is defined as the minimal work of a derivation from $u_1$ to $u_2$. If $n_1$ and $n_2$ are natural numbers we define the \textit{(derivational) work of $\mathcal{P}$ at $(n_1, n_2)$} as
  $$\Omega(n_1, n_2)=\textsf{max}\{\textsf{work}(u_1, u_2)\, : \, (u_1, u_2) \, \,\textrm{is a relation}, \, \, |u_1|\leqslant n_1, |u_2|\leqslant n_2\}.$$
          \end{definition}

The function \textsf{work} has been introduced by J.-C. Birget in (\cite{Birg}, 1.4).

We can define the work of a relation in terms of van Kampen diagrams (see Section \ref{VKdfsp}). Indeed the work of a defining relation is equal to the number of edges of the van Kampen diagram associated with it. Let $(u_0, u_1, \cdots, u_k)$ be a $k$ steps derivation and let $\mathcal{C}$ be the van Kampen diagram associated with it; then the work of $(u_0, u_1, \cdots, u_k)$ is equal to the number of edges of $\mathcal{C}$ belonging to a single face plus twice the number of edges of $\mathcal{C}$ belonging to two faces.

   \begin{remark} \label{smile} \rm Let $h$ be the minimal length\footnote{we recall that if $(a,b)\in R$, then the length of $(a,b)$ is $|a|+|b|$} of elements of $R$; then for every relation $(a,b)$ we have that $h \textsf{dl}(a,b)\leqslant \textsf{work}(a,b)$. This implies that 
                     $$h \textsf{DL}\leqslant \Omega.$$
      Let the length of the elements of $R$ be bounded above by a constant $h'$ (in particular let $R$ be finite); then $\textsf{work}(a,b)\leqslant h' \textsf{dl}(a,b)$ and $\Omega \leqslant h' \textsf{DL}$, thus 
             $$h \textsf{DL} \leqslant \Omega \leqslant h' \textsf{DL};$$ 
 therefore in this case $\textsf{DL}$ and $\Omega$ are strongly equivalent (Definition \ref{equiv}).      
       In particular if all the elements of $R$ have the same length and if $h$ is this length then $\textsf{work}(a,b)= h \textsf{dl}(a,b)$ and $\Omega = h \textsf{DL}$.
       \end{remark}

   \begin{remark} \label{} \rm Let $\mathcal{P}:=< X\, | \, R >$ be a group presentation and let $\mathcal{P}'$ be the monoid presentation corresponding to $\mathcal{P}$ (Definition \ref{segr}). In (\cite{MO2}, Def. 3.1), Madlener and Otto have introduced a function, which we call \textit{strong derivational length} and denote \textsf{sdl}, in the following way: 
           $$\textsf{sdl}(1, xx^{-1})=1$$
if $x\in X\cup X^{-1}$ and 
   $$\textsf{sdl}(u, v)=1+|v|$$
if $(u, v)$ is a defining relation not equal to $(1, xx^{-1})$ for every $x \in X\cup X^{-1}$. Then $\textsf{sdl}$ is extended to the whole set of relations in the same way as this is done for the function $\textsf{work}$.

It is easy to see that for every relation $(a,b)$ we have the following inequality
            $$\textsf{dl}(a,b) \leqslant \textsf{sdl}(a,b) \leqslant \textsf{dl}(a,b)+ \textsf{work}(a,b).$$
   Analogously we define for natural numbers $n_1, n_2$         
      $$\textsf{SDL}(n_1, n_2):=\textsf{max}\{\textsf{sdl}(u_1,u_2): (u_1, u_2)   \, \, \mathrm{is \, \,  a \, \, relation},  \,|u_1|\leqslant n_1, \,  |u_2|\leqslant n_2\}$$       
     and by the preceding inequality we have that       
        $$\textsf{DL} \leqslant \textsf{SDL} \leqslant \textsf{DL} + \Omega.$$
     With the same argument of Proposition \ref{gersho} one proves that $\mathcal{P}$ has a solvable Word Problem if and only if \textsf{SDL} is computable.     
      \end{remark}

\begin{problem} \label{} Find an upper bound for $\Omega$ in function of $\textsf{SDL}$.
         \end{problem}

Given a non-negative rational number $n$ we denote by $\lfloor n \rfloor$ \textit{the integer part of $n$}, that is the biggest natural number less or equal to $n$.

\begin{exercise} \label{vugar} Prove that $\textsf{DL}(m, n)=\lfloor m/2 \rfloor+\lfloor n/2 \rfloor$ for the standard monoid presentation of a free group, that is for the presentation $< X\cup X^{-1} \, | \, \texttt{F} >$ where $\texttt{F}$ is as in Definition \ref{stop}.
         \end{exercise}

\begin{exercise} \label{mov} Let $\mathcal{CM}(X)$ be the free commutative monoid on $X$ and consider the presentation given in Definition \ref{comm}. Let $m, n$ be natural numbers such that $m\geqslant n$. Prove that: \begin{enumerate}

\item $\displaystyle \textsf{DL}(m, n)=\frac{(n-1)n}{2}$ \, if $|X| \geqslant n$ (in particular if $X$ is infinite);

\item $\displaystyle \textsf{DL}(m, n)=h(n-h) + \frac{(h-1)h}{2}$ \, if $|X|=h+1<n$.
       \end{enumerate}  \end{exercise}

By Remark \ref{smile} we have that $\Omega = 2 \textsf{DL}$ for the presentation of Exercise \ref{vugar} and $\Omega = 4 \textsf{DL}$ for that of Exercise \ref{mov}.

\begin{proposition} \label{Derv} Let $\mathcal{P}:=< X\, | \, R >$ be a monoid presentation, let $u$ be a word on $X$ and let $k$ be a natural number. \begin{enumerate}

  \item If $\mathcal{P}$ is finitely generated decidable then the set of derivations from $u$ of work equal to $k$ is finite and effectively computable;

  \item if $\mathcal{P}$ is finitely related then the set of $k$ steps derivations from $u$ is finite and effectively computable.
   \end{enumerate}
         \end{proposition}

\begin{proof} \begin{enumerate}
 
  \item Let $X$ be finite and $R$ decidable. Let $h \leqslant k$, let $n\leqslant \textsf{min}\{|u|, h\}$ and find all the subwords of $u$ of length $n$. If $a$ is any of these subwords, consider all the words on $X$ of length up to $k-|a|$; if $b$ is any of these words decide whether $(a, b)$ is a defining relation. In the affirmative case consider the 1-step derivation from $u$ by means of $(a, b)$. All the 1-step derivations from $u$ of work up to $k$ are obtained this way; moreover since $X$ is finite, the number of these words $b$ is finite and effectively computable.
  
   If $v$ is a word obtained from $u$ by a derivation of work $k$, then there exist words $u_0=u, u_1, \cdots, u_m=v$ such that $(u_i, u_{i+1})$ is a 1-step derivation of work $k_i$ and $k_1+\cdots+k_m=k$. Since these derivations can be all obtained with the procedure described above, then the set of derivations from $u$ of work equal to $k$ is finite and effectively computable.

  \item Let $\mathcal{P}$ be finitely related and let $q=|R|$; it is sufficient to prove that the set of 1-step derivations from $u$ is finite and effectively computable. For every $(a, b)\in R$, check whether $a$ is a subword of $u$ and for any of these cases apply the relation $(a, b)$ to $u$. There are thus at most $q|u|$ words to compute and this proves the claim. 
   \end{enumerate}
             \end{proof}

\begin{lemma} \label{istar} Let $\mathcal{P}$ be an enumerably related monoid presentation with solvable Word Problem and let $\textsf{DL}$ and $\Omega$ be its derivational length and work. If $\mathcal{P}$ is finitely generated then $\textsf{DL}$ and $\Omega$ are bounded above by computable functions; if $\mathcal{P}$ is finitely generated decidable then $\Omega$ is computable; if $\mathcal{P}$ is finite then $\textsf{DL}$ is computable.
\end{lemma}

\begin{proof} Since the set of generators and that of defining relators are enumerable, then by Remark \ref{spesp} the solvability of the Word Problem implies that of the Word Search Problem. Let $\mathcal{P}$ be finitely generated and let $n_1$ and $n_2$ be naturals; then the number of pairs of words of length bounded by $n_1$ and $n_2$ respectively is finite. For any of these pairs of words we solve the Word Search Problem; this gives for any relation $(u_1,u_2)$ with $|u_1|\leqslant n_1$ and $|u_2|\leqslant n_2$ an upper bound $f_{(u_1,u_2)}$ for $\textsf{dl}(u_1,u_2)$ and an upper bound $g_{(u_1,u_2)}$ for $\textsf{work}(u_1,u_2)$. We can compute in a finite time the maxima of these bounds which are thus upper bounds for $\textsf{DL}(n_1,n_2)$ and $\Omega(n_1,n_2)$.

Let $\mathcal{P}$ be finitely generated decidable, let $n_1$ and $n_2$ be natural numbers and let $(u_1,u_2)$ be a relation with $|u_1|\leqslant n_1$ and $|u_2|\leqslant n_2$. By 1 of Proposition \ref{Derv}, for any $k=1, \cdots, g_{(u_1,u_2)}$ the derivations from $u$ of work equal to $k$ are finitely many and effectively computable. For at least one of these $k$'s, the word $v$ belongs to the set of derivations from $u$ of work equal to $k$ and the minimal $k$ for which this happens is equal to the derivational work from $u$ to $v$. Therefore we can compute $\Omega$ in a finite time and thus $\Omega$ is computable.

Let $\mathcal{P}$ be finite; then the proof that $\textsf{DL}$ is computable is the same as the preceding case with the difference that by 1 of Proposition \ref{Derv}, for any $k=1, \cdots, f_{(u_1,u_2)}$ the $k$ steps derivations from $u$ are finitely many and effectively computable.
      \end{proof}

A converse of Lemma \ref{istar} holds:

\begin{proposition} \label{gersho} \begin{enumerate}
  
  \item A finitely generated decidable monoid presentation has a solvable Word Problem if and only if $\Omega$ is computable, if and only if $\Omega$ is bounded above by a computable function; 
  
    \item a finite monoid presentation has a solvable Word Problem if and only if $\textsf{DL}$ is computable, if and only if $\textsf{DL}$ is bounded above by a computable function.\end{enumerate}
          \end{proposition}

\begin{proof}  We give the proof for 1, that for 2 being analogous.

By Lemma \ref{istar}, if the Word Problem is solvable for $\mathcal{P}$ then $\Omega$ is computable which implies trivially that $\Omega$ is bounded above by a computable function. We have to prove that if $\Omega$ is bounded above by a computable function then the Word Problem is solvable.

Let $f:\mathbb{N}^2 \rightarrow \mathbb{N}$ be a computable function such that $\Omega(n_1, n_2)\leqslant f(n_1, n_2)$ and let $u_1,u_2$ be words such that $|u_1|\leqslant n_1$ and $|u_2|\leqslant n_2$. We have that $(u_1,u_2)$ is a relation if and only if there exists a derivation from $u_1$ to $u_2$ of work less or equal to $f(n_1, n_2)$. By 1 of Proposition \ref{Derv} these derivations can be computed in a finite time, thus one can decide whether $(u_1,u_2)$ is a relation. This solves the Word Problem.  
     \end{proof}

Part 2 of Proposition \ref{gersho} was first proved (in a less general form) by Madlener and Otto in (\cite{MO}, Lem. 3.2).

\section{The Word Problem for group presentations}

We recall that $\rho$ denotes the reduced form (Definition \ref{rho}).

\begin{definition} \label{defarea} \rm Let $<  X \, | \, R  >$ be a group presentation, let $\mathcal{N}$ be the normal subgroup of $\mathcal{F}(X)$ normally generated by $R$ and let $w \in \mathcal{N}$. The \textit{area of $w$} is defined as
       $$\textsf{area}(w)=\textsf{min}\{k\in \mathbb{N} : \rho(a_1 r_1  a_1^{-1} \cdots a_k  r_k  a_k^{-1})=w, \, \, a_i \in \mathcal{F}(X), \, \, r_i \in R\}.$$
       The \textit{group work of $w$} is defined as
    $$\textsf{gwork}(w)=\textsf{min}\{|r_1|+\cdots+|r_n|  :  \rho(a_1 r_1 a_1^{-1} \cdots a_k r_k a_k^{-1})=w, \, \, a_i \in \mathcal{F}(X), \, \, r_i \in R\}.$$
 \end{definition}

The area of a relator was introduced by M. Gromov in (\cite{Grom}, Sec. 2.3).

Given a relator $w$, its area is equal to the minimal number of faces of van Kampen diagrams for $w$. Given a van Kampen diagram, we define its \textit{work} as the number of edges belonging to a single face plus twice the number of edges belonging to two faces. Then the group work of $w$ is the minimal work of a van Kampen diagram for $w$.

 \begin{remark} \label{smer} \rm Let $\mathcal{P}:=\langle \, X \, | \, R \, \rangle$ be a group presentation, let $w$ be a relator of $\mathcal{P}$ and let $r_1, \cdots, r_k$ be defining relators and $a_1, \cdots, a_k$ be words such that $w$ is the reduced form of 
               $$a_1 r_1  a_1^{-1} \cdots a_k r_k u_k^{-1}.$$
Let $c$ be the maximal length of $r_1, \cdots, r_k$ and let $|w|=n$. By using the properties of the associated van Kampen diagram one can prove that there exist words $b_1, \cdots, b_k$ such that $|b_i|\leqslant ck+n$ and $w$ is the reduced form of 
              $$b_1 r_1  b_1^{-1} \cdots b_k r_k b_k^{-1},$$
that is we can choose the conjugating elements of length no more than $ck+n$. See Prop. 2.2 of \cite{Gerst} or The. 1.1 and 2.2 of \cite{Short} for a proof of this fact\footnote{In (\cite{LS}, Rem. after Lem. V.1.2) it is given $kn$ as bound for the length of the conjugating elements, but $kn$ is worse than $ck+n$ since $c$ is a constant.} (see also Lem. 7.1 of \cite{Miller}).
    \end{remark}

\begin{definition} \label{Dehn} \rm Let $\mathcal{P}$ be a finitely generated group presentation and let $n$ be a natural number; the \textit{Dehn function} of $\mathcal{P}$ at $n$ is
    $$\Delta(n):=\textsf{max}\{\textsf{area}(w) : w \, \,\textrm{is a relator and} \, |w|\leqslant n\}.$$
    The \textit{group work} of $\mathcal{P}$ at $n$ is
 $$\Omega_g(n)=\textsf{max}\{\textsf{gwork}(w)\, : \, w \, \,\textrm{is a relator and} \, \, |w|\leqslant n\}.$$
    \end{definition}

 In \cite{GI}, R. I. Grigorchuk and S. V. Ivanov have defined a function $f_1$ from $\mathbb{N}$ to $\mathbb{N}$ in the following way: if $n$ is a natural number then $f_1(n)$ is the minimal number of edges of van Kampen diagrams for relators of $\mathcal{P}$ of length less or equal to $n$.

     \begin{exercise} \label{}  Prove that $f_1$ is equivalent (Definition \ref{equiv}) to $\Omega_g$.      \end{exercise}

       \begin{remark} \label{smile2} \rm In the same way as shown in Remark \ref{smile} we have that if $h$ is the minimal length of elements of $R$ then for every relator $w$ we have that $h \, \textsf{area}(w)\leqslant \textsf{gwork}(w)$, thus 
                $$h \Delta \leqslant \Omega_g;$$
and if the length of the elements of $R$ is bounded above by a constant $h'$ (in particular if $R$ is finite) then $\textsf{area}(w) \leqslant h' \textsf{gwork}(w)$ and $\Omega_g \leqslant h' \Delta$, thus 
       $$h \Delta \leqslant \Omega_g \leqslant h' \Delta.$$ 
 Therefore in this case $\Delta$ and $\Omega_g$ are strongly equivalent (Definition \ref{equiv}).      
       \end{remark}

\begin{exercise} \label{mov2} Let $\mathcal{FA}(X)$ be the free abelian group on $X$ and consider the presentation given in Definition \ref{comm}. Prove that: \begin{enumerate}

\item $\displaystyle \Delta(2n+1)=\Delta(2n)=\frac{(n-1)n}{2}$ \, if $|X| \geqslant n$ (in particular if $X$ is 

infinite);

\item $\displaystyle \Delta(2n+1)=\Delta(2n)=h(n-h) + \frac{(h-1)h}{2}$ \, if $|X|=h+1<n$.
       \end{enumerate}  \end{exercise}

By Remark \ref{smile2} we have that $\Omega_g = 4 \Delta$ for the presentation of Exercise \ref{mov2}.

\begin{remark} \label{sabb} \rm Let $\mathcal{P}=<  X \, | \, R  >$ be a group presentation and let us consider the monoid presentation corresponding to $\mathcal{P}$ (Definition \ref{segr}). Let $u$ and $v$ be non-necessarily reduced words on $X\cup X^{-1}$ such that $(u, v)$ is a 1-step derivation. Thus either $v$ is obtained from $u$ by means of a cancellation or by the application of a relation $(r, 1)$ where $r \in R$. In the first case $\rho(u)=\rho(v)$, in the second there exists a word $a$ such that $\rho(v)=\rho(ara^{-1} \,u)$, thus $\rho(u)=\rho(ar^{-1}a^{-1} \, v)$. This implies that if $v$ is obtained from $u$ by a derivation of length $k$, then there exists $h \in \{0, 1, \cdots, k\}$ and there exist words $a_1, \cdots, a_h$ and $r_1, \cdots, r_h \in R$ such that $\rho(v)=\rho(a_1 r_1 a_1^{-1} \cdots a_h r_h a_h^{-1} \, u)$.
 
 Let $u=1$, that is $v$ is a relator; if $k$ is the derivational length of $(1, v)$, then there exists an expression of $v$ as product of $h\leqslant k$ conjugates of defining relators, that is the area of $v$ is less or equal to the derivational length from 1 to $v$. In formula, $\textsf{area}(v)\leqslant \textsf{dl}(v,1)$ and thus $\Delta(n) \leqslant \textsf{DL}(n,0)$ for every natural $n$.
 
 In the same way one proves that $\textsf{gwork}(v)\leqslant \textsf{work}(v,1)$ and $\Omega_g(n) \leqslant \Omega(n,0)$.    \end{remark}

   \begin{proposition} \label{opti} Let $\mathcal{P}:=\langle \, X \, | \, R \, \rangle$ be a group presentation, let $\Delta$ and $\Omega_g$ be the Dehn function and the group work of $\mathcal{P}$ and let $\textsf{DL}$ and $\Omega$ be the be the derivational length and work of the monoid presentation corresponding to $\mathcal{P}$. Let $n$ be a natural number; we have that \begin{enumerate}
    
  \item $\Omega_g(n) \leqslant \Omega(n,0) \leqslant 4\Omega_g^3(n) + \Omega_g^2(n) + (4n+1)\Omega_g(n) - n;$
  
  \item let the length of the elements of $R$ be bounded above by a natural $c$ (in particular let $\mathcal{P}$ be finitely related); then
  $$\Delta(n) \leqslant \textsf{DL}(n,0) \leqslant 2c \Delta^2(n) + \left(2n+\frac{c}{2}+1\right)\Delta(n) - \frac{n}{2}$$ 
          and  
  $$\Omega_g(n) \leqslant \Omega(n,0) \leqslant 4c \,\Omega_g^2(n) + (4n+c+1)\Omega_g(n) - n.$$
\end{enumerate} 
        \end{proposition}

     \begin{proof} The inequalities $\Delta(n) \leqslant \textsf{DL}(n,0)$ and $\Omega_g(n) \leqslant \Omega(n,0)$ have been proved in Remark \ref{sabb}. First we prove 1 and the second inequality of 2. 
 
 Let $w$ be a relator of length $n$ and set $h=\textsf{gwork}(w)$. Let  
             $$w=\rho(a_1 r_1  a_1^{-1} \cdots a_k r_k a_k^{-1})$$
 with $|r_1|+ \cdots+ |r_k|=h$. Let $p=\textsf{max}\{|r_1|, \cdots, |r_k|\}$; by Remark \ref{smer} there exist words $b_1, \cdots, b_k$ such that $|b_i|\leqslant p k+n$ and $w$ is the reduced form of 
              $$u:=b_1 r_1  b_1^{-1} \cdots b_k r_k b_k^{-1}.$$
We show that in the monoid presentation corresponding to $\mathcal{P}$ there exists a derivation from 1 to $w$ of work less or equal to $4h^3+ h^2  + (4n+1)h - n$ for the inequality of 1 and of work less or equal to $4ch^2+ (4n+c+1)h - n$ for the second inequality of 2.

Since $|b_i|\leqslant pk+n$ for $i=1, \cdots, k$, then there is a derivation from 1 to $b_1  b_1^{-1} \cdots b_k b_k^{-1}$ of work less or equal to $2(pk+n)k$, thus there is a derivation from 1 to $u$ of work less or equal $h+2(pk+n)k=h+2pk^2+2nk$.

Since $|r_i| \leqslant p$, then $|b_i r_i b_i^{-1}| \leqslant p+2(pk+n)$ and thus 
            $$|u|\leqslant k[2(pk+n)+p]=2pk^2 + (2n+p)k.$$

We have that $w$ is the reduced form of $u$, thus there is a derivation from $u$ to $w$ consisting only in cancellations, that is of work 
     $$|u|-|w| \leqslant 2pk^2 + (2n+p)k - n.$$

In conclusion there is a derivation from 1 to $w$ of work less or equal to 
    \begin{equation} \label{r&f} h+2pk^2+2nk+  2pk^2 + (2n+p)k - n=h+4pk^2  + (4n+p)k - n \leqslant  \end{equation} 
    $$\leqslant  4ph^2  + (4n+p+1)h - n$$
 where the last inequality follows from the fact that $k\leqslant h$.    
 
Since $p\leqslant h$ then we have that (\ref{r&f}) is less or equal to $4h^3 + h^2 + (4n+1)h - n$ and this proves 1.

Let the length of the elements of $R$ be bounded above by a natural $c$; then this implies that $p \leqslant c$ and then that (\ref{r&f}) is less or equal to $4ch^2  + (4n+c+1)h - n$ and this proves the second inequality of 2.

For the first inequality of 2 the procedure is analogous with the following modifications. We take $k$ as $\textsf{area}(w)$; the derivation from 1 to $u$ has length less or equal $pk^2+(n+1)k$; there is a derivation from $u$ to $w$ consisting only in cancellations and of length less or equal to 
       $$\frac{|u|-|w|}{2} \leqslant pk^2 + \left(n+\frac{p}{2}\right)k - \frac{n}{2}.$$  \end{proof}

 \begin{problem} \label{} Find better bounds than those given in Proposition \ref{opti} for $\textsf{DL}(n,0)$ in function of $\Delta(n)$ and for $\Omega(n,0)$ in function of $\Omega_g(n)$; or prove that the latter are optimal by finding group presentations in which these bounds are attained. 
 \end{problem}

We have the following
        
\begin{proposition} \label{} Let $\mathcal{P}:=< X\, | \, R >$ be a group presentation and let $n$ and $k$ be natural numbers. \begin{enumerate}

\item If $\mathcal{P}$ is finitely generated decidable then the set of relators of length and group work equal to $n$ and $k$ respectively is finite and effectively computable;

  \item if $\mathcal{P}$ is finite then the set of relators of length and area equal to $n$ and $k$ respectively is finite and effectively computable.
  
     \end{enumerate}
         \end{proposition}

  \begin{proof} \begin{enumerate}
  
  \item Let $w$ be a relator of length and group work equal to $n$ and $k$ respectively; then $w$ is the reduced form of a product of at most $k$ conjugates of defining relators the sum of whose lengths is equal to $k$. Since $\mathcal{P}$ is finitely generated then the words of length up to $k$ are finitely many; since $\mathcal{P}$ is decidably related then one can decide which of these words are defining relators. In particular since none of these has length more than $k$, then by Remark \ref{smer} we have that the conjugating elements can be chosen of length no more than $k^2+n$, that is we can find all the relators of length and work equal to $n$ and $k$ in a finite time.
  
  \item Let $c$ be the maximal length of defining relators; then a relator of area and length equal to $k$ and $n$ respectively is the reduced form of a product of $k$ conjugates of defining relators with the length of any conjugating element bounded above by $ck+n$ by Remark \ref{smer}. Thus all the relators of length and area equal to $n$ and $k$ can be found in a finite time.
\end{enumerate}
     \end{proof}

 The proofs of the next two propositions are analogous to those of Lemma \ref{istar} and Proposition \ref{gersho}.

\begin{lemma} \label{} Let $\mathcal{P}$ be an enumerably related group presentation with solvable Word Problem and let $\Delta$ and $\Omega_g$ be its Dehn function and group work. If $\mathcal{P}$ is finitely generated then $\Delta$ and $\Omega_g$ are bounded above by computable functions; if $\mathcal{P}$ is finitely generated decidable then $\Omega_g$ is computable; if $\mathcal{P}$ is finite then $\Delta$ is computable.
\end{lemma}

\begin{proposition} \label{gersho2} \begin{enumerate}
  
  \item A finitely generated decidable group presentation has a solvable Word Problem if and only if $\Omega_g$ is computable, if and only if $\Omega_g$ is bounded above by a computable function; 
  
    \item a finite group presentation has a solvable Word Problem if and only if $\Delta$ is computable, if and only if $\Delta$ is bounded above by a computable function.\end{enumerate}
          \end{proposition}

 Part 2 of Proposition \ref{gersho2} is no longer true if the number of relators is infinite even when it is decidable: see Example 2.4 of \cite{GI}, where it is shown a finitely generated decidable presentation with unsolvable Word Problem and Dehn function constantly equal to 2.

 \section{The Conjugacy Problem}

\begin{definition} \label{} \rm Let $M$ be a monoid and let $a, b\in M$; we say that $a$ is \textit{conjugated to $b$} if there exists $t\in M$ such that $t a=b t$.     \end{definition}

The conjugacy is a reflexive and transitive relation; it is also symmetric if $M$ is a group.

\begin{definition} \label{} \rm Let $\mathcal{P}=<  X \, | \, R  >$ be a presentation for a monoid $M$. We say that $\mathcal{P}$ \textit{has a solvable Conjugacy Problem} if given $a, b\in \mathcal{M}(X)$ there exists an algorithm which determines whether the element of $M$ determined by $a$ is conjugated to that determined by $b$, that is whether or not there exists a word $t\in \mathcal{M}(X)$ such that $(t a, b t)$ is a relation for $\mathcal{P}$. This is equivalent to say that the set
      \begin{equation} \label{sargi}  \{(a, b)\in \mathcal{M}(X)^2 \,: \,\, t a= b t\,\, \mathrm{in} \,\, M \,\, \textrm{for some} \,\,  t \in \mathcal{M}(X)\} \end{equation}
is decidable. 

We say that $\mathcal{P}$ \textit{has a solvable Conjugacy Search Problem} if it has a solvable Conjugacy Problem and if there exists an algorithm which finds a conjugating word $t$.      \end{definition}

\begin{remark} \label{} \rm Let $\mathcal{P}=<  X \, | \, R  >$ be a presentation for a monoid $M$ and let the Conjugacy Problem be solvable. If $M$ is a group then also the Word Problem is solvable since a word on $X$ is equal to 1 in $M$ if and only if it is conjugated to 1. If $M$ is not a group then the Word Problem is not necessarily solvable since there can be an algorithm which, given two words $a$ and $b$, decides whether $a$ is conjugated to $b$ but not if $a=b$ in $M$.  \end{remark}

\begin{problem} \label{prob1} Find a presentation with solvable Conjugacy Problem but unsolvable Word Problem or prove that the solvability of the Conjugacy Problem implies that of the Word Problem. \end{problem}

\begin{remark} \label{} \rm Let $\langle R \rangle$ be the set of relations of $\mathcal{P}$. If $\mathcal{M}(X)$ and $\langle R \rangle$ are enumerable (in particular if $X$ and $R$ are enumerable, see Corollary \ref{avig}), then the set (\ref{sargi}) is enumerable. This is because in this case the set $\mathcal{M}(X)^3 \times \langle R \rangle$ is enumerable and thus is equal to a set of the form 
     $$\{\alpha_n=(a_n, b_n, t_n, u_n, v_n) : n\in \mathbb{N}, \, a_n, b_n, t_n \in \mathcal{M}(X), \, (u_n, v_n) \in \langle R \rangle\},$$
      where there is an algorithm whose $n$-th output is $\alpha_n$. Consider the algorithm whose $n$-th output is $(a_n, b_n)$ if $t_n a_n = u_n$ and $b_n t_n = v_n$, otherwise the output is empty; this algorithm enumerates (\ref{sargi}).   \end{remark}

\begin{remark} \label{CSP} \rm Let $\mathcal{P}=<  X \, | \, R  >$ be a presentation such that $\mathcal{M}(X)$ and $\langle R\rangle$ are enumerable. Then the solvability of the Conjugacy Problem is equivalent to that of the Conjugacy Search Problem. Indeed   
  $$\mathcal{M}(X) \times \langle R\rangle=\{\beta_n=(t_n, u_n, v_n) : n\in \mathbb{N}, \, t_n \in \mathcal{M}(X), (u_n, v_n)\in \langle R\rangle\},$$
where there is an algorithm whose $n$-th output is $\beta_n$. Let $a$ be conjugated to $b$; let us consider the algorithm whose $n$-th output is $t_n$ if $u_n=t_n a$ and $v_n=b t_n$, otherwise the output is empty. Then for some $n$ the output will be non-empty and $t_n$ will be a conjugating word.    \end{remark}

We have seen in Remark \ref{spesp} that the complexity of the Word Search Problem can be much greater than that of the Word Problem. It would be interesting to give an answer to the analogous problem for the Conjugacy Problem.

\begin{problem} Find a presentation whose Conjugacy Search Problem has a complexity greater than that of the Conjugacy Problem; or prove that the two complexities are equal for any presentation. \end{problem}

\begin{definition} \label{gamma} \rm Let $M$ be a monoid, let $X$ be a generating set for $M$ and let $a_1, a_2\in \mathcal{M}(X)$ be such that the element of $M$ determined by $a_1$ is conjugated to that determined by $a_2$. We set
       $$\gamma(a_1, a_2):=\textsf{min}\{|t| \,\, : t \in \mathcal{M}(X) \,\,\mathrm{and} \,\, ta_1=a_2 t \,\, \textrm{in} \,\, M \}.$$
  Let $X$ be finite and let $n_1, n_2$ be natural numbers: we set
       $$\Gamma(n_1, n_2):=\textsf{max}\{ \gamma(a_1, a_2) \,\, : a_1 \,\, \mathrm{is \,\, conjugated \,\, to} \,\, a_2 \,\, \mathrm{in} \,\, M$$
        $$\mathrm{and} \,\, |a_1|\leqslant n_1, \,\, |a_2|\leqslant n_2\}.$$    \end{definition}

\medskip

The functions $\gamma$ and $\Gamma$ depend on the generating set $X$. If $n'_1\leqslant n_1$ and $n'_2\leqslant n_2$ then $\Gamma(n'_1, n'_2) \leqslant \Gamma(n_1, n_2)$. We have seen that if $M$ is a group then the conjugacy is a symmetric relation, thus $\gamma$ and $\Gamma$ are symmetric functions. If $M$ is not a group then this is not true anymore, that is there can be two elements $a_1$ and $a_2$ such that $a_1$ is conjugated to $a_2$ but $a_2$ is not conjugated to $a_1$. Or it may happen that $a_1$ and $a_2$ are conjugated one to the other but $\gamma(a_1, a_2)\neq \gamma(a_2, a_1)$. Thus $\Gamma(n_2, n_1)$ is not necessarily equal to $\Gamma(n_1, n_2)$.

 \begin{definition} \label{} \rm  Let $\mathcal{P}=<  X \, | \, R  >$ be a monoid presentation; the function \textit{$\Gamma$ relative to $\mathcal{P}$} is the function $\Gamma$ with respect to $X$ for the monoid presented by $\mathcal{P}$. \end{definition}

 \begin{definition} \label{} \rm  Let $u,v \in\mathcal{M}(X)$; we say that $u$ and $v$ are \textit{cyclic conjugates} one of the other if there exist word $a,b \in\mathcal{M}(X)$ such that $u=ab$ and $v=ba$. \end{definition}

The cyclic conjugation is an equivalence relation.

\begin{exercise} \label{gashi} Prove that: \begin{enumerate}

  \item two elements of $\mathcal{M}(X)$ are conjugate if and only if they are cyclic conjugates;
  
  \item let $a$ and $b$ be words such that $ab\neq ba$; then $\gamma(ab, ba)=|b|$ and $\gamma(ba, ab)=|a|$, thus $\gamma(ab, ba)\neq \gamma(ba, ab)$ if $|a|\neq |b|$;

  \item $\Gamma(m,n)=\Gamma(n,m)=\lfloor \textsf{min}(m,n)/2 \rfloor$ for $\mathcal{M}(X)$.
  
\end{enumerate}
       \end{exercise}

Part 1 of Exercise \ref{gashi} is a solution of the Conjugacy Problem for the free monoid.

 \begin{problem} \label{prob2} \begin{enumerate}
 
  \item Find an example in which $a_1$ is conjugated to $a_2$ but not the contrary.
  
  \item Find a monoid $M$ and two natural numbers $n_1$ and $n_2$ such that $\Gamma(n_2, n_1)\neq \Gamma(n_1, n_2)$.
  \end{enumerate}
  
     Or prove that some or all of the preceding points are impossible. \end{problem}

\smallskip

Let $M$ be a \textit{cancellative monoid}, that is for every $a,b,t \in M$ if $ta=tb$ or $at=bt$ then $a=b$. Any group is a cancellative monoid. Let $M$ be generated by a set $X$, let $a$ be a word on $X$ such that the element of $M$ represented by $a$ is central. If $b$ is another word on $X$ then $a$ is conjugated to $b$ if and only if $a$ and $b$ determine the same element of $M$; in this case $\gamma(a, b)=0$. Thus if $M$ is a commutative cancellative monoid then $\Gamma$ is equal to the zero function. The converse is also true, that is a commutative monoid is cancellative if and only if $\Gamma=0$. If $M$ is a group then it is also true that $M$ is abelian if and only if $\Gamma=0$. These results hold with respect to every generating set of $M$.

If $M$ is a finite monoid then $\Gamma$ is bounded. Indeed if $X$ is a generating set for $M$ and if $h$ is the maximal length of elements of $M$ then $\Gamma$ is eventually equal to a constant less or equal to $h$.

On the other side, let $M$ be generated by a finite set $X$ and let $\Gamma$ be bounded; then this implies that there exists a finite subset $\{t_1, \cdots, t_n\}$ of $M$ such that for every $a,b\in M$ such that $a$ is conjugated to $b$ there exists $i=1, \cdots, n$ such that $t_i a=b t_i$.

\begin{problem} Find a finitely generated infinite non-commutative cancellative monoid whose function $\Gamma$ (with respect to a finite generating set) is bounded; or prove that a commutative monoid for which $\Gamma$ is bounded is finite or cancellative. Solve the same problem for groups instead that for cancellative monoids or prove that a group whose function $\Gamma$ is bounded is finite or abelian. \end{problem}

\begin{definition} \label{} \rm  Let $\mathcal{F}(X)$ be the free group on $X$ and let $u\in \mathcal{F}(X)$, that is $u$ is a reduced word on $X\cup X^{-1}$. We say that $u$ is \textit{cyclically reduced} if the first and the last letters of $u$ are not inverse one of the other. If $u$ is a reduced word which is not cyclically reduced then there exist unique reduced words $a$ and $u'$ such that $u'$ is cyclically reduced and $u=a u' a^{-1}$. The word $u'$ is called the \textit{cyclically reduced form of $u$}. \end{definition}

\begin{definition} \label{} \rm  Let $G$ be a group and let $X$ be a subset of $G$ which generates $G$ as a group. Let $u, v \in \mathcal{F}(X)$; then $u$ and $v$ are conjugated in $G$ if and only if their cyclically reduced forms are conjugated. Let $a_1, a_2$ be cyclically reduced words in $\mathcal{F}(X)$ which are conjugated in $G$; then we set
         $$\gamma_0(a_1, a_2):=\textsf{min}\{|t| \,\, : t \in \mathcal{F}(X) \,\,\mathrm{and} \,\, ta_1=a_2 t \,\, \textrm{in} \,\, G \}.$$
    Let $n_1, n_2$ be natural numbers: we set
      $$\Gamma_0(n_1, n_2):=\textsf{max}\{ \gamma_0(a_1, a_2) \,\, : a_1 \,\, \mathrm{and} \,\, a_2 \,\, \mathrm{are \,\, conjugated \,\, in} \,\, G $$
        $$ \mathrm{and} \,\, |a_1|\leqslant n_1, \,\, |a_2|\leqslant n_2\}.$$ \end{definition}

\smallskip

The Conjugacy Problem for $\mathcal{F}(X)$ is solvable; indeed by The. 1.3 of \cite{MKS} we have that two cyclically reduced words are conjugated in $\mathcal{F}(X)$ if and only if they are cyclic conjugates. 

We recall that the functions $\gamma$ and $\Gamma$ are symmetric for a group.

\begin{exercise} \label{} Prove that if $G=\mathcal{F}(X)$ then $\Gamma_0(m,n)=\lfloor n/2 \rfloor+1$ for $m\geqslant n$.
       \end{exercise}

It would be interesting to compute the function $\Gamma_0$ for group presentations with a solvable Conjugacy Problem. For instance for a finite presentation $\mathcal{P}$ satisfying the small cancellation condition $C'(1/8)$ then $\Gamma_0(n_1, n_2)\leqslant 6 \textsf{max}(n_1, n_2)$. If $\mathcal{P}$ satisfies $C'(1/6)$ then $\Gamma_0(n_1, n_2)\leqslant 12 \textsf{max}(n_1, n_2)$ and if $\mathcal{P}$ satisfies $C'(1/4)$ and $T(4)$ then $\Gamma_0(n_1, n_2)\leqslant 16 \textsf{max}(n_1, n_2)$ (see The. V.5.4 of \cite{LS}). More generally, if $\mathcal{P}$ is a finite presentation of an hyperbolic group, then there exists a constant $k$ such that $\Gamma_0(n_1, n_2)\leqslant k \textsf{max}(n_1, n_2)$ (see III-$\Gamma$-2.11 and 2.12 of \cite{BriHae}).

\smallskip

\begin{proposition} \label{CP} Let $\mathcal{P}=<  X \, | \, R  >$ be a finitely generated presentation with solvable Word Problem and whose set of relations is enumerable (in particular let $R$ be enumerable). $\mathcal{P}$ has a solvable Conjugacy Problem if and only if $\Gamma$ is bounded above by a computable function. In this case $\Gamma$ is computable. \end{proposition}

\begin{proof} Let $M$ be the monoid presented by $\mathcal{P}$. \begin{enumerate}

  \item Let $\Gamma(n_1, n_2)\leqslant f(n_1, n_2)$ where $f(n_1, n_2)$ is computable. Since $\mathcal{P}$ is finitely generated, then for any naturals $n_1$ and $n_2$ the number of pairs of words $(a_1, a_2)$ such that $|a_1|\leqslant n_1$ and $|a_2|\leqslant n_2$ is finite. Let $a_1, a_2$ be words such that $|a_1|\leqslant n_1$ and $|a_2|\leqslant n_2$; $a_1$ is conjugated to $a_2$ if and only there exists a word $t$ of length no more than $f(n_1, n_2)$ such that $(t a_1, a_2 t)$ is a relation. 

Let us solve the Word Problem for all the pairs of the form $(t a_1, a_2 t)$ where $|t|\leqslant f(n_1, n_2)$; this solves the Conjugacy Problem for $(a_1, a_2)$. Moreover if $a_1$ is conjugated to $a_2$ then the minimal length of a conjugating word is equal to $\gamma(a_1, a_2)$ and thus $\Gamma(n_1, n_2)$ is computable since it is the maximum of the $\gamma(a_1, a_2)$.
  
  \item Let $\mathcal{P}$ be have a solvable Conjugacy Problem. Since the set of relation of $\mathcal{P}$ is enumerable then by Remark \ref{CSP}, $\mathcal{P}$ has a solvable Conjugacy Search Problem. Let $n_1, n_2$ be natural numbers; since $\mathcal{P}$ is finitely generated, then the number of pairs of words $(a_1, a_2)$ such that $|a_1|\leqslant n_1$ and $|a_2|\leqslant n_2$ is finite. Let us solve the Conjugacy Search Problem for these pairs of words; this gives for any $(a_1, a_2)$ such that $a_1$ is conjugated to $a_2$ an upper bound for $\gamma(a_1, a_2)$. The maximum of these bounds is an upper bound for $\Gamma(n_1, n_2)$ which is then bounded above by a computable function.  \end{enumerate}
\end{proof}

\begin{corollary} \label{} Let $\mathcal{P}$ be a finitely generated presentation whose set of relators is enumerable and suppose that $\mathcal{P}$ presents a group. Then $\mathcal{P}$ has a solvable Conjugacy Problem if and only if $\mathcal{P}$ has a solvable Word Problem and $\Gamma$ is computable. \end{corollary}

\section{The Membership Problem}

\begin{definition} \label{} \rm Let $\mathcal{P}=< X \, | \, R >$ be a presentation for a monoid $M$ and let $A$ be a subset of $\mathcal{M}(X)$. We say that $\mathcal{P}$ has \textit{has a solvable Membership Problem on $A$} if given $w\in \mathcal{M}(X)$ there exists an algorithm which determines whether the element of $M$ represented by $w$ belongs to the submonoid of $M$ generated by $A$, that is whether there exist elements $a_1, \cdots, a_m\in A$ such that $(w, a_1 \cdots a_m)$ is a relation for $\mathcal{P}$. This is equivalent to say that the set
      \begin{equation} \label{ituzi}  \{(w, a_1, \cdots, a_m) : w \in \mathcal{M}(X), \, a_i \in A \,\, \textrm{and} \,\, w=a_1 \cdots a_m \,\, \mathrm{in} \,\, M\} \end{equation}
is decidable.  
   
   We say that $\mathcal{P}$ \textit{has a solvable Membership Search Problem on $A$} if it has a solvable Membership Problem on $A$ and if there exists an algorithm which finds elements $a_1, \cdots, a_m$ of $A$ such that $w=a_1 \cdots a_m$ in $M$. \end{definition}

The Membership Problem is also called \textit{generalized word problem} or \textit{uniform word problem}.

 If $\mathcal{P}$ presents a group then the solvability of the Membership Problem on $A=\{1\}$ is equivalent to that of the Word Problem.

\begin{problem} Find a monoid presentation with solvable Membership Problem on $A=\{1\}$ but unsolvable Word Problem; or prove that the solvability of the Membership Problem on $A=\{1\}$ implies that of the Word Problem. \end{problem}

We refer the reader to the introduction of \cite{MMS} for a list of groups for which the Membership Problem is known to be solvable or not.

\begin{remark} \label{} \rm If $\mathcal{M}(X)$, $\langle R \rangle$ and $A$ are enumerable, then the set (\ref{ituzi}) is enumerable. Indeed in this case the Cartesian product of $\mathcal{M}(X)$, $\langle R \rangle$ and the set of finite sequences of elements of $A$ is enumerable, that is it is equal to
     $$\{\alpha_n=(w_n, a_1, \cdots, a_{k_n}, u_n, v_n) : n\in \mathbb{N}, \, w_n \in \mathcal{M}(X), \, a_i \in A, \,(u_n, v_n)\in \langle R \rangle \},$$
    where there is an algorithm whose $n$-th output is $\alpha_n$. Consider the algorithm whose $n$-th output is $(w_n, a_1, \cdots, a_{k_n})$ if $w_n = u_n$ and $v_n=a_1 \cdots a_{k_n}$, otherwise the output is empty; this algorithm enumerates (\ref{ituzi}).  \end{remark}

\begin{remark} \label{icel} \rm If $A$ and $\langle R \rangle$ are enumerable, then the solvability of the Membership Problem on $A$ is equivalent to that of the Membership Search Problem on $A$. Indeed the Cartesian product of $\langle R \rangle$ and the set of finite sequences of elements of $A$ is enumerable and thus is equal to
     $$\{\beta_n=(a_1, \cdots, a_{k_n}, u_n, v_n) : n\in \mathbb{N}, \, a_i \in A, \, (u_n, v_n)\in \langle R \rangle\},$$ 
   where there is an algorithm whose $n$-th output is $\beta_n$. Let $w$ belong to the submonoid generated by $A$ and consider the algorithm whose $n$-th output is $(a_1, \cdots, a_{k_n})$ if $w = u_n$ and $a_1 \cdots a_{k_n}=v_n$, otherwise the output is empty. Then necessarily for some $n$ the output will be non-empty.  \end{remark}

   \begin{problem} Find a presentation whose Membership Search Problem on a some set $A$ has a complexity greater than that of the Membership Problem on $A$; or prove that the two complexities are equal for any presentation and any set $A$. \end{problem}

\begin{definition} \label{dist} \rm Let $\mathcal{P}=<  X \, | \, R  >$ be a presentation for the monoid $M$, let $A$ be a subset of $\mathcal{M}(X)$, let $T$ be the submonoid of $M$ generated by $A$ and let $t\in T$. Then there exist $a_1, \cdots, a_m\in A$ such that $t=a_1 \cdots a_m$. The least of such $m$ is called \textit{the length of $t$ in $A$} and denoted $\textsf{length}_A(t)$.
 
 Let $n$ be a natural number: the \textit{distortion at $n$ of $\mathcal{P}$ relative to $A$} is
     $$\textsf{Dist}_A(n)=\textsf{max}\{\textsf{length}_A(t) : t \in \mathcal{M}(X), \, |t| \leqslant n, $$
     $$t=a_1 \cdots a_m \, \textrm{in } M \, \,\textrm{for some}  \,a_1, \cdots, a_m\in A \}.$$   \end{definition}

This definition of distortion function is due B. Farb (\cite{Farb}, Sec. 2). A slightly different notion of distortion was previously been given by Gromov (\cite{Grom2}, Chap. 3).

\begin{remark} \label{} \rm  Let $\mathcal{P}=<  X \, | \, R  >$ be a presentation for the monoid $M$, let $T$ be a submonoid of $M$ and let $A$ and $B$ be subsets of $\mathcal{M}(X)$ which generate $T$. Set $h:=\sup \{\textsf{length}_B(a) : a \in A\}$ and $h':=\sup \{\textsf{length}_A(b) : b \in B\}$. If $h$ and $h'$ are finite then 
           $$\textsf{length}_B(t) \leqslant h \, \textsf{length}_A(t) \leqslant h  h' \, \textsf{length}_B(t)$$
 and thus
    $$\textsf{Dist}_B \leqslant h \, \textsf{Dist}_A \leqslant h h' \, \textsf{Dist}_B,$$
      that is $\textsf{Dist}_A$ and $\textsf{Dist}_B$ are strongly equivalent (Definition \ref{equiv}). We observe that if $A$ and $B$ are finite then $h$ and $h'$ are finite.

We also observe that if $A$ is finite then there exists a finite subset $B'$ of $B$ such that $T$ is generated by $B'$. Indeed for any $a\in A$ there exist $b_1, \cdots, b_n$ elements of $B$ such that $a$ is a monomial in $b_1, \cdots, b_n$. Then the subset of $B$ whose elements are these $b_i$ for all the $a\in A$ is a finite subset of $B$ and generates $T$ since $A$ does.
      \end{remark}

  \begin{lemma} \label{carl} Let $\mathcal{P}=<  X \, | \, R  >$ be a presentation for a monoid $M$, let $A$ be a finite subset of $\mathcal{M}(X)$ and suppose that given a word on $X$ and one on $A$ it is decidable whether these two words are equal in $M$ (in particular let the Word Problem be solvable for $\mathcal{P}$). If $\textsf{Dist}_A$ is bounded above by a computable function then $\mathcal{P}$ has a solvable Membership Problem; if moreover $\mathcal{P}$ is finitely generated then $\textsf{Dist}_A$ is computable. \end{lemma}

\begin{proof} Let $T$ be the submonoid of $M$ generated by $A$, let $\textsf{Dist}_A(n)\leqslant f(n)$ where $f(n)$ is a computable function, let $u$ be a word on $X$ and let $m=|u|$. Then $u$ belongs to $T$ if and only there exist $a_1, \cdots, a_k\in A$ with $k\leqslant f(m)$ such that $u=a_1 \cdots a_k$ in $M$. Thus one can decide whether $u$ belongs to $T$ by deciding whether $u= a_1 \cdots a_k$ in $M$ for all $a_1, \cdots, a_k\in A$ and $k\leqslant f(m)$. Since $A$ is finite and $f$ is computable then this is done in finite time.

Let $\mathcal{P}$ be finitely generated. Then for every natural $n$ there are finitely many words on $X$ of length $n$. Let $u$ be one of these words and decide whether $u=v$ in $M$ for every $v\in T$ such that $\textsf{length}_A(v) \leqslant f(n)$. If $u$ belongs to $T$ then $\textsf{length}_A(u)$ is the least of those $\textsf{length}_A(v)$ when $u=v$. Thus $\textsf{length}_A$ is computable and therefore $\textsf{Dist}_A$ is computable. \end{proof}

  \begin{proposition} \label{MP} Let $\mathcal{P}=<  X \, | \, R  >$ be a finitely generated presentation for a monoid $M$ and let the set of relations of $\mathcal{P}$ be enumerable. Let $A$ be a finite subset of $\mathcal{M}(X)$ and suppose that given a word on $X$ and one on $A$ it is decidable whether these two words are equal in $M$ (in particular let the Word Problem be solvable for $\mathcal{P}$). Then $\mathcal{P}$ has a solvable Membership Problem on $A$ if and only if $\textsf{Dist}_A$ is computable. \end{proposition}

\begin{proof} Let $T$ be the submonoid of $M$ generated by $A$. By Lemma \ref{carl} we have that if $\textsf{Dist}_A$ is computable then the Membership Problem on $A$ is solvable.

  Let $\mathcal{P}$ have a solvable Membership Problem; by Remark \ref{icel}, $\mathcal{P}$ has a solvable Membership Search Problem. Let $n$ be a natural number; since $\mathcal{P}$ is finitely generated, then the number of words of length bounded above by $n$ is finite. If $u$ is one of these words solve the Membership Search Problem on $u$. If $u$ belongs to $T$ then this gives an expression of $u$ as monomial on $A$ and thus it gives an upper bound for $\textsf{langth}_A(u)$. The maximum of these bounds is an upper bound for $\textsf{Dist}_A$ which is then bounded above by a computable function. With the same argument of the proof of Lemma \ref{carl} one proves that $\textsf{Dist}_A$ is computable.  \end{proof}

The first proof of Proposition \ref{MP} has been given in (\cite{MMS}, Prop. 1.1), but earlier a slightly less general statement has been proved in (\cite{Farb}, Prop. 2.1).

\section{Tietze transformations} \label{tt}

\begin{definition} \label{Transf} \rm Let $\mathcal{P}=<  X \, | \, R  >$ be a monoid presentation and let us consider the following transformations on $\mathcal{P}$ which consist in adding or deleting “superfluous” generators and defining relations. They are called \textit{Tietze transformations}: 

\begin{enumerate}

  \item Let $Q\subset \langle  R \rangle$; then replace $R$ with $R \cup Q$, that is replace $<  X \, | \, R  >$ with $< X \, | \, R \cup Q >$.
  
  \item Let $Q\subset R$ be such that $\langle R\setminus{Q} \rangle=\langle R \rangle$;  then replace $R$ with $R\setminus{Q}$, that is replace $<  X \, | \, R  >$ with $< X \,\, | \, \, R\setminus{Q} >$.

  \item Let $U \subset \mathcal{M}(X)$ and for every $u\in U$ let $y_u$ be an element not belonging to $\mathcal{M}(X)$ and such that if $u$ and $u'$ are distinct elements of $U$ then $y_u\neq y_{u'}$. Set $Y:=\{y_u : u \in U\}$ and $T:=\{(y_u, u): u \in U\}$; then replace $X$ with $X\cup Y$ and $R$ with $R \cup T$, that is replace $<  X \, | \, R  >$ with $<  X \cup Y \, | \, R \cup T >$.
  
  \item Let $Y\subset X$ and suppose that for every $y \in Y$ there exists a word $u_y$ not containing $z$ for every $z \in Y$ and such that $(y, u_y) \in R$. Let $\varphi$ be the homomorphism from $\mathcal{M}(X)$ to $\mathcal{M}(X \setminus{Y})$ which sends any $y \in Y$ to $u_y$ and any $x \in X\setminus{Y}$ to itself. Then set $T:=\{(y, u_y) : y \in Y\}$, replace $X$ with $X\setminus{Y}$ and replace $R$ with $V=\{\left(\varphi(a), \varphi(b)\right) : (a, b) \in R\setminus{T}\}$, that is replace $<  X \, | \, R  >$ with $< X \setminus{Y} \, | \, V >$.\end{enumerate}

  Let $\mathcal{P}=<  X \, | \, R  >$ be a group presentation and let us denote by $\langle \langle R \rangle \rangle$ the set of relators of $\mathcal{P}$, that is the normal subgroup of $\mathcal{F}(X)$ normally generated by $R$. The \textit{Tietze transformations} for a group presentation are defined in the same way but with the following modifications: in 1 and 2 replace $\langle R \rangle$ and $\langle R\setminus{Q} \rangle$ with $\langle \langle R \rangle \rangle$ and $\langle \langle R\setminus{Q} \rangle \rangle$ respectively\footnote{we observe that $\langle R \rangle$ is a subset of $\mathcal{M}(X)^2$ while $\langle \langle R \rangle \rangle$ is a subset of $\mathcal{F}(X)$.}. In 3, $U \subset \mathcal{F}(X)$, $y_u$ is an element not belonging to $\mathcal{F}(X)$ and $T:=\{y^{-1}_u u: u \in U\}$. In 4, for every $y \in Y$ there exists a word $u_y$ not containing neither $z$ nor $z^{-1}$ for every $z \in Y$ and such that $y^{-1} u_y \in R$; $\varphi$ is the homomorphism from $\mathcal{F}(X)$ to $\mathcal{F}(X \setminus{Y})$ which sends any $y \in Y$ to $u_y$ and any $x \in X\setminus{Y}$ to itself; finally we replace $R$ with $\varphi(R\setminus{T})$.  \end{definition}

The transformations 1 and 2 add or delete respectively superfluous defining relations; transformation 3 and 4 add or delete superfluous generators and defining relations. Transformations of types 1 and 2 are inverse one of the other; the inverse of a transformations of type 3 is a transformation of type 4. The inverse of a transformations of type 4 is a transformation of type 3, followed by one of type 1 and by one of type 2; this is because first we have to add $Y$ to the generators and $T$ to the defining relators obtaining $< X \, | \, T \cup V >$, then add $R$ and finally delete $V$ from the set of defining relations.

We say that two presentations are \textit{of the same kind} if they are both monoid or group presentations.

It is obvious that by applying Tietze transformations to a presentation we obtain a presentation for the same monoid. The converse is also true, that is if $\mathcal{P}$ and $\mathcal{P}'$ are presentations of the same kind for the same monoid then there exists a finite sequence of Tietze transformations which applied to $\mathcal{P}$ gives $\mathcal{P}'$ (see\footnote{Th. 1.5 of \cite{MKS} is proved for group presentations but the same argument holds also for monoid presentations with the obvious modifications.} Th. 1.5 of \cite{MKS}). In fact by the proof of Th. 1.5 of \cite{MKS} we have the following stronger result:

\begin{proposition} \label{Tietze} Two presentations of the same kind for the same monoid can be obtained one from the other by applications of at most four Tietze transformations, namely a transformation of type 3 followed by one of type 1, then by one of type 2 and finally by one of type 4 (some of these transformations can be empty). \end{proposition}

\begin{definition} \label{} \rm A Tietze transformation is said \textit{finite} if it adds or deletes only finitely many generators or defining relations, that is if the sets $Q$, $T$ and $Y$ of Definition \ref{Transf} are finite\footnote{In the literature an \textit{elementary Tietze transformation} is defined as a Tietze transformation which adds or deletes exactly one generator and/or defining relation. An elementary Tietze transformation is obviously finite and any finite Tietze transformation is obtained by repeated applications of elementary Tietze transformations.}. In the same way a Tietze transformation is said \textit{decidable} if the sets $Q$, $T$ and $Y$ are decidable. \end{definition}

By\footnote{The proof of Cor. 1.5 of \cite{MKS} is given for elementary Tietze transformations but the same argument holds also for finite and decidable Tietze transformations} Cor. 1.5 of \cite{MKS} we have that two finite [respectively decidable] presentations of the same kind define the same monoid if and only if one can be obtained from the other by repeated applications of finite [respectively decidable] Tietze transformations. Moreover Proposition \ref{Tietze} holds also in the special cases of finite or decidable presentations, that is two finite or decidable presentations of the same kind for the same monoid can be obtained one from the other by applications of at most four Tietze transformations of the types specified in Proposition \ref{Tietze}.

\smallskip

 Let us consider a graph whose vertices are all the monoid presentations and with an edge joining two presentations if one can be obtained from the other by application of a Tietze transformation; any connected component of this graph corresponds to an isomorphism class of monoids. Furthermore by Proposition \ref{Tietze}, two vertices are either non-connected or they are connected by a path of length at most four.

 Moreover let $M$ be a monoid and let $\mathcal{P}$ and $\mathcal{P}'$ be two finite monoid presentations of $M$. Then there is a path (of length at most four) from $\mathcal{P}$ to $\mathcal{P}'$ whose edges are finite [respectively decidable] Tietze transformations.

The same thing can be said if we take as vertices of the graph all the group presentations.

  \begin{definition} \label{btt} \rm  We say that a Tietze transformation of type 1 or 2 is \textit{of bounded derivation} if it adds or deletes a set of relations whose derivational length is bounded, that is if $\textsf{sup}\{\textsf{dl}(q) : q\in Q\}$ is finite where $Q$ is as in 1 and 2 of Definition \ref{Transf}. Analogously we define Tietze transformations of type 1 or 2 \textit{of bounded work}, \textit{of bounded area} and of \textit{of bounded group work} by replacing $\textsf{dl}$ by $\textsf{work}$, $\textsf{area}$ and $\textsf{gwork}$ respectively\footnote{of course if we talk about Tietze transformations of bounded derivation or work then this means that we are dealing with a monoid presentation, while for Tietze transformations of bounded area or group work the presentation is a group presentation}.

    We say that a Tietze transformation of type 3 or 4 is of \textit{bounded length} if it adds or deletes a set of generators whose length is bounded, that is if $\textsf{sup}\{|u| : u\in U\}$ or $\textsf{sup}\{|u_y| : y\in Y\}$ (respectively for 3 or 4 of Definition \ref{Transf}) are finite. 
    
 A Tietze transformation of bounded derivation, work, area or group work is said a \textit{bounded} Tietze transformation.    \end{definition}

 A finite Tietze transformation is necessarily bounded. If the presentation is finitely generated then any Tietze transformation of bounded length is necessarily finite. But if the presentation is finitely related then a Tietze transformation of bounded derivation, work, area or group work needs not be finite.

   \begin{remark} \label{stef} \rm Let $\mathcal{P}$ be a presentation and let $\mathcal{P}'$ be the presentation obtained from $\mathcal{P}$ by application of a bounded Tietze transformation. We want to determine the relationship between the derivational length, the work, the area, the group work and the length of group elements in $\mathcal{P}$ and $\mathcal{P}'$.

   Suppose that $\mathcal{P}'$ is obtained from $\mathcal{P}$ by application of a Tietze transformation of type 1 of bounded derivation and let $c$ be the maximal derivational length of defining relations added. Let $\textsf{dl}$ and $\textsf{DL}$ the functions “derivational length” relative to $\mathcal{P}$ (see Definition \ref{derlen}) and let $\textsf{dl}'$ and $\textsf{DL}'$ those relative to $\mathcal{P}'$. Then we have for a relation $(u, v)$ that
              $$\textsf{dl}'(u, v) \leqslant \textsf{dl}(u, v) \leqslant c \, \textsf{dl}'(u, v)$$
  and then that
      \begin{equation} \label{T1} \textsf{DL}' \leqslant \textsf{DL} \leqslant c \, \textsf{DL}'.\end{equation}
      
 In the same way one proves that if $\mathcal{P}'$ is obtained from $\mathcal{P}$ by application of a Tietze transformation of type 1 of work, area or group work bounded by a constant $c$ then
          $$\Omega' \leqslant \Omega \leqslant c \, \Omega',$$
          $$\Delta' \leqslant \Delta \leqslant c \, \Delta',$$
and    
       $$\Omega_g' \leqslant \Omega_g \leqslant c \, \Omega_g',$$
where $\Omega$, $\Delta$ and $\Omega_g$ are the function “work”, the Dehn function and the function “group work” relative to $\mathcal{P}$ and $\Omega'$, $\Delta'$ and $\Omega'_g$ those relative to $\mathcal{P}'$.

 Let us now suppose that $\mathcal{P}'$ is obtained from $\mathcal{P}$ by application of a Tietze transformation of type 3 of bounded length. We will use the notation of 3 of Definition \ref{Transf}. There exists a natural number $c$ such that $|u|\leqslant c$ for every $u\in U$. Let $\varphi$ be the function from $\mathcal{M}(X\cup Y)$ to $\mathcal{M}(X)$ sending any element of $X$ to itself and any $y_u\in Y$ to $u$. We have that $\varphi$ is the identity on $\mathcal{M}(X)$, in particular $\big(\varphi(u), \varphi(v)\big)=(u,v)$ for every $(u,v) \in R$, and that $\big(\varphi(y_u), \varphi(u)\big)=(u,u)$ for every $(u, y_u) \in T$. Moreover $|\varphi(w)|\leqslant c |w|$.

 If $\langle R \rangle$ and $\langle R\cup T \rangle$ are the congruences of $\mathcal{M}(X)$ and $\mathcal{M}(X\cup Y)$ generated by $R$ and $R \cup T$ respectively, then we have that $\langle R \rangle \subset \langle R\cup T \rangle$ and that $\langle R \rangle=\{\big(\varphi(u), \varphi(v)\big) : (u,v) \in \langle R\cup T \rangle\}$. Let $(a, b)$ be a 1-step derivation by means of $(u,v) \in R\cup T$; if $(u,v) \in R$ then $\big(\varphi(u), \varphi(v)\big)$ is a 1-step $R$-derivation, if $(u,v) \in T$ then $\varphi(u)=\varphi(v)$.

   Let $(u, v)$ be an $R\cup T$-relation and let $(u=u_0, u_1, \cdots, u_{k-1}, v=u_k)$ be an $R\cup T$-derivation where $k=\textsf{dl}'(u, v)$. Then there is an $R$-derivation from $\varphi(u)$ to $\varphi(v)$ of length at most $k$, that is
       $$\textsf{dl}\big(\varphi(u), \varphi(v)\big)\leqslant \textsf{dl}'(u,v).$$
    If $(u, v)$ is an $R$-relation then since $\big(\varphi(u), \varphi(v)\big)=(u,v)$ and since obviously $\textsf{dl}'(u, v) \leqslant \textsf{dl}(u, v)$ then $\textsf{dl}(u, v)=\textsf{dl}'(u, v)$. In particular we have that 
                      $$\textsf{DL} \leqslant \textsf{DL}'.$$
    Let $(u, v)$ be an $R\cup T$-relation but not an $R$-relation. There exist $R\cup T$-derivations from $u$ to $\varphi(u)$ and from $v$ to $\varphi(v)$ of lengths at most $|u|$ and $|v|$. Since there is an $R$-derivation from $\varphi(u)$ to $\varphi(v)$ of length $\textsf{dl}\big(\varphi(u), \varphi(v)\big)$, this gives an $(R\cup T)$-derivation from $u$ to $v$ of length $\textsf{dl}\big(\varphi(u), \varphi(v)\big) + |u| +|v|$, that is 
             $$\textsf{dl}'(u, v)\leqslant \textsf{dl}\big(\varphi(u), \varphi(v)\big) + |u| +|v|$$ 
     and then since $|\varphi(u)|\leqslant c|u|$ and $|\varphi(v)|\leqslant c|v|$ we have that 
                     $$\textsf{DL}'(m, n)\leqslant \textsf{DL}(c m, c n)+m+n,$$ 
that is 
          \begin{equation} \label{T3} \textsf{DL}(m, n) \leqslant \textsf{DL}'(m, n)\leqslant \textsf{DL}(c m, c n)+m+n. \end{equation}

   With an analogous argument one proves that 
 $$\Omega(m, n) \leqslant \Omega'(m, n)\leqslant \Omega(c m, c n)+(c+1)(m+n),$$
   $$\Delta(n) \leqslant \Delta'(n)\leqslant \Delta(c n)+n$$
and
    $$\Omega_g(n) \leqslant \Omega_g'(n)\leqslant \Omega_g(cn)+(c+1)n.$$

  Let us now determine the relationship between the length of elements in $\mathcal{P}$ and $\mathcal{P}'$. If $M$ is the monoid presented by $\mathcal{P}$ and $\mathcal{P}'$ and if $a$ is any element of $M$ then 
          \begin{equation} \label{hagen} |a|' \leqslant |a| \leqslant c|a|'\end{equation}  
where $|\cdot|$ and $|\cdot|'$ denote respectively the lengths with respect to the sets of generators of $\mathcal{P}$ and $\mathcal{P}'$ respectively. \end{remark}

 \begin{remark} \label{total} \rm Let $\mathcal{P}:=< X \, |  \, R >$ be a presentation such that $X$ and $\langle R \rangle$ are enumerable and let $\mathcal{P}'$ be the presentation obtained from $\mathcal{P}$ by adding to $R$ all the relations of $\mathcal{P}$, that is $\mathcal{P}'=< X \, |  \, \langle R \rangle >$.
 
The set of defining relations of $\mathcal{P}'$ is not finite, is enumerable and is decidable if and only if the Word Problem is solvable for $\mathcal{P}$. Moreover there exist natural numbers $m_0$ and $n_0$ such that $\textsf{DL}(m,n)=1$ and $\Omega(m,n)=m+n$ for every $m\geqslant m_0$ and $n\geqslant n_0$ (if $m< m_0$ and/or $n< n_0$ then either $\textsf{DL}(m,n)=1$ and $\Omega(m,n)=m+n$ or $\textsf{DL}(m,n)=\Omega(m,n)=0$). \end{remark}

     \begin{proposition} \label{dasz} Let $\mathcal{P}$ and $\mathcal{P}'$ be presentations such that $\mathcal{P}'$ is obtained from $\mathcal{P}$ by applications of Tietze transformations of bounded derivation or length (in particular, let $\mathcal{P}$ and $\mathcal{P}'$ be finite presentations for the same monoid). Then the derivational lengths for $\mathcal{P}$ and $\mathcal{P}'$ are equivalent. \end{proposition}

\begin{proof} By $(\ref{T1})$ and $(\ref{T3})$, we have that by applying a Tietze transformation of type 1 or 3, the equivalence class of the derivational length does not change. This is also true by applying a Tietze transformation of type 2 because its inverse is a transformation of type 1; and is true also for a Tietze transformation of type 4 since its inverse is a transformation of type 3 followed by one of type 1 and by one of type 2.    \end{proof}

 \begin{remark} \label{dasz2} \rm Let $\mathcal{P}'$ be obtained from $\mathcal{P}$ by applications of Tietze transformations of bounded length and bounded work, area or group work. Then the work, the Dehn function or the group work of $\mathcal{P}$ and $\mathcal{P}'$ respectively are equivalent. The proof is analogous to that of Proposition \ref{dasz}.

 We observe that Grigorchuk and Ivanov have proved (\cite{GI}, The. 1.6) that the group work is invariant up to equivalence under some transformations called \textit{T-transformations} and \textit{stabilizations}. We have improved their result since the latter are special cases of Tietze transformations of bounded group work.
 \end{remark}

\begin{remark} \label{chuck} \rm If $\mathcal{P}'$ is obtained from $\mathcal{P}$ by applying non-bounded Tietze transformations then $\textsf{DL}$ and $\textsf{DL}'$ are in general non-equivalent. For instance if $\mathcal{P}$ and $\mathcal{P}'$ are as in Remark \ref{total}, then $\textsf{DL}$ and $\textsf{DL}'$ are equivalent if and only if $\textsf{DL}$ is equivalent to a constant function. In particular if $\mathcal{P}$ is a finite presentation with unsolvable Word Problem then $\textsf{DL}$ is greater than any computable function while $\textsf{DL}'$ is constant. This shows that if in Proposition \ref{dasz} we remove the hypothesis of boundedness for the Tietze transformations applied to a presentation, then the derivational lengths of two presentations of the same kind for the same monoid can be not related at all.

This shows also that a decidable Tietze transformation needs not be of bounded derivation (or bounded work, or bounded area). Indeed let $\mathcal{P}$ and $\mathcal{P}'$ as in Remark \ref{total}, let $\mathcal{P}$ be finitely generated decidable, let $\textsf{DL}$ be not equivalent to a constant function and suppose that the Word Problem is solvable for $\mathcal{P}$. Then $\mathcal{P}'$ is decidable, thus one can obtain $\mathcal{P}'$ from $\mathcal{P}$ by applications of decidable Tietze transformations; but since $\textsf{DL}$ and $\textsf{DL}'$ are not equivalent then $\mathcal{P}'$ cannot be obtained from $\mathcal{P}$ by applications of Tietze transformations of bounded derivation.
  \end{remark}

 Let $\mathcal{P}=< X \, | \, R >$ be a monoid presentation, let $A\subset \mathcal{M}(X)$ and suppose that we apply to $\mathcal{P}$ a Tietze transformation $\tau$. Then \textit{the set corresponding to $A$ by $\tau$} is equal to $A$ if $\tau$ is a Tietze transformation of type 1, 2 or 3; if $\tau$ is a Tietze transformation of type 4, \textit{the set corresponding to $A$ by $\tau$} is equal to $\varphi(A)$ where $\varphi$ is as in 4 of Definition \ref{Transf}. 
 
 If $\mathcal{P}'=< X' \, | \, R' >$ is the presentation obtained from $\mathcal{P}$ by application of a series of Tietze transformations, then \textit{the set corresponding to $A$ in $\mathcal{P}'$} is the set obtained applying to $A$ these transformations. In particular it is a subset of $\mathcal{M}(X')$.

\begin{proposition} \label{new} Let $\mathcal{P}=< X \, | \, R >$ and $\mathcal{P}'=< X' \, | \, R' >$ be presentations such that $\mathcal{P}'$ is obtained from $\mathcal{P}$ by applications of Tietze transformations of type 3 or 4 of bounded length and by any Tietze transformations of type 1 or 2 (in particular, let $\mathcal{P}$ and $\mathcal{P}'$ be finitely generated presentations for the same monoid). \begin{enumerate}

  \item Let $A\subset \mathcal{M}(X)$ and let $A'$ be the set corresponding to $A$ in $\mathcal{P}'$. Then the distortion functions (Definition \ref{dist}) relative to $A$ and $A'$ are equivalent.
  
  \item If $\Gamma$ and $\Gamma'$ are the functions of Definition \ref{gamma} relative to $\mathcal{P}$ and $\mathcal{P}'$ respectively then $\Gamma$ and $\Gamma'$ are strongly equivalent (Definition \ref{equiv}).
  
\end{enumerate}
  \end{proposition}

\begin{proof} Let $M$ be the monoid presented by $\mathcal{P}$ and $\mathcal{P}'$ and let $\textsf{Dist}_A$ and $\textsf{Dist}'_{A'}$ be the distortion functions relative to $A$ and $A'$. Since $\Gamma$ and the distortion function is independent from the defining relations of a presentation, then they do not change by applying Tietze transformations of type 1 or 2 . Let us apply a Tietze transformation of length bounded by a natural number $c$; let $\gamma$ and $\gamma'$ be the functions of Definition \ref{gamma} relative to $\mathcal{P}$ and $\mathcal{P}'$ respectively and let $\textsf{length}_A$ and $\textsf{length}'_{A'}$ be the length functions relative to $A$ and $A'$. By (\ref{hagen}) we have that 
  $$\gamma'(w_1, w_2) \leqslant \gamma(w_1, w_2) \leqslant c\gamma'(w_1, w_2)$$
 for every $w_1, w_2\in \mathcal{M}(X)$ such that $w_1$ is conjugated to $w_2$ and
      $$\textsf{length}'_{A'}(t) \leqslant \textsf{length}_A(t) \leqslant c\textsf{length}'_{A'}(t)$$
 for every $t$ belonging to the submonoid generated\footnote{obviously the submonoid generated by $A$ is equal to that generated by $A'$} by $A$. Thus  
          $$\Gamma' \leqslant \Gamma \leqslant c \, \Gamma'$$
 and 
          $$\textsf{Dist}'_{A'}(n) \leqslant \textsf{Dist}_A(n) \leqslant \textsf{Dist}'_{A'}(cn)$$
 and the claim is proved.  \end{proof}
                  
The first proof of 1 of Proposition \ref{new} has been given in (\cite{Farb}, Prop. 3.1), see also Prop. 1.2 of \cite{MMS}.

 \appendix 
 
  \section{A paradox of Computability Theory} \label{app}

  Let $E$ be an enumerable set and suppose that every finite subset of $E$ (i.e., every subset of $E$ with finitely many elements) is decidable. We prove that this implies that every subset of $E$ is decidable, while by admitting the \textit{Church-Turing thesis} (see 3.3 of \cite{Sipser}) there are subsets of $\mathbb{N}$ which are not even enumerable.
  
  Let $F$ be any subset of $E$. Since $E$ is enumerable then $E=\{e_n : n \in \mathbb{N}\}$, where there is an algorithm whose $n$-th output is $e_n$. Let $F_n=\{e \in F : e=e_k \, \, \mathrm{for  \, \, some  \, \,} k\leqslant n \}$. Each $F_n$ is finite and by our hypothesis it is decidable.
  
  Let $e$ be an element of $E$; then there exists $n \in \mathbb{N}$ such that $e=e_n$. We have that $e$ belongs to $F$ if and only if $e$ belongs to $F_n$. But $F_n$ is decidable, thus there exists an algorithm which decides whether $e$ belongs to $F_n$, thus $F$ is decidable.





\vspace{5cm}

Addresses: 

\textit{Institut de Mathématiques de Jussieu}

\textit{175, rue du Chevaleret}

\textit{75013 Paris - France}        

and

\textit{Dipartimento di Matematica dell'Università di Palermo}

\textit{Via Archirafi 34}

\textit{90123 Palermo - Italy}

\smallskip \smallskip

e-mail: \textsf{carmelovaccaro@yahoo.it}

 \end{document}